\documentclass{birkjour}
 \newtheorem{thm}{Theorem}[section]
 \newtheorem{cor}[thm]{Corollary}
 \newtheorem{lem}[thm]{Lemma}
 \newtheorem{prop}[thm]{Proposition}
 \theoremstyle{definition}
 
 \theoremstyle{remark}
 \newtheorem{rem}[thm]{Remark}
 
 \numberwithin{equation}{section}

\begin{document}

%
%
%
%
%
%
%
%
%

\title[Sobolev inequalities on homogeneous groups]
 {Sobolev type inequalities, Euler-Hilbert-So\-bo\-lev and Sobolev-Lorentz-Zygmund spaces on homogeneous groups}

\author[M. Ruzhansky]{Michael Ruzhansky}

\address{%
 Department of Mathematics\\
Imperial College London\\
180 Queen's Gate, London SW7 2AZ\\
United Kingdom}

\email{m.ruzhansky@imperial.ac.uk}

\thanks{The authors were supported in parts by the EPSRC
 grants EP/K039407/1 and EP/R003025/1, and by the Leverhulme Grants RPG-2014-02 and RPG-2017-151. The second author was also supported by the MESRF No 02.a03.21.0008 and the MESRK grant AP05130981. The third author was supported by the MESRK grant AP05133271. No new data was collected or generated during the course of research.}
\author[D. Suragan]{Durvudkhan Suragan}
\address{%
 Institute of Mathematics and Mathematical Modelling\\
125 Pushkin str.\\
050010 Almaty\\
Kazakhstan\\
and\\ RUDN University\\
6 Miklukho-Maklay str., Moscow 117198\\
Russia}
\email{suragan@math.kz}

\author[N. Yessirkegenov]{Nurgissa Yessirkegenov}

\address{%
Institute of Mathematics and Mathematical Modelling\\
125 Pushkin str.\\
050010 Almaty\\
Kazakhstan\\
and\\
 Department of Mathematics\\
Imperial College London\\
180 Queen's Gate, London SW7 2AZ\\
United Kingdom}

\email{n.yessirkegenov15@imperial.ac.uk}


\date{January 17, 2018}

\begin{abstract}
We define Euler-Hilbert-Sobolev spaces and obtain embedding results on homogeneous groups using Euler operators, which are homogeneous differential operators of order zero. Sharp remainder terms of $L^{p}$ and weighted Sobolev type and Sobolev-Rellich inequalities on homogeneous groups are given. Most inequalities are obtained with best constants. As consequences,
     we obtain analogues of the generalised classical
     Sobolev type and Sobolev-Rellich inequalities. We also discuss applications of logarithmic Hardy inequalities to Sobolev-Lorentz-Zygmund spaces.
     The obtained results are new already in the anisotropic $\mathbb R^{n}$ as well as in the isotropic $\mathbb R^{n}$ due to the freedom in the choice of any homogeneous quasi-norm.
\end{abstract}

\maketitle
\tableofcontents
\section{Introduction}

In this paper we are interested in Hardy, Poincar\'{e}, Sobolev,  Rellich and higher order inequalities of Sobolev-Rellich type in the setting of general homogeneous groups. Furthermore, we are interested in questions of best constants, their attainability, and sharp expressions for the remainders. As consequences, we define Euler-Hilbert-Sobolev and Sobolev-Lorentz-Zygmund spaces on homogeneous groups. For a short review in this direction and some further discussions we refer to recent papers \cite{Ruzhansky-Suragan:critical, Ruzhansky-Suragan:Layers, Ruzhansky-Suragan:Hardy, Ruzhansky-Suragan:identities, Ruzhansky-Suragan:uncertainty} and \cite{ORS16} as well as to references therein. In the case of $\mathbb{R}^{n}$ expressions for the remainder terms in Hardy and Rellich inequalities have been recently analysed in \cite{MOW:Hardy-Hayashi, IIO, IIO:Lp-Hardy, MOW-Rellich}.


The obtained results yield new statements already in the Euclidean setting of $\mathbb{R}^{n}$ but when we are working with anisotropic differential structures. Moreover, even in the isotropic situation in $\mathbb{R}^{n}$, the novelty of the obtained results is also in the arbitrariness of the choice of any homogeneous quasi-norm, since most of the inequalities are obtained with best constants. In this situation, the very convenient framework for working with a given dilation structure is that of homogeneous groups as was initiated in the book \cite{FS-Hardy} of Folland and Stein.
One of the ideas there is a distillation of results of harmonic analysis depending only on the group and dilation structures, and we believe that this paper contributes to this direction.  Our results hold on general anisotropic $\mathbb{R}^{n}$,
including the Heisenberg group and other nilpotent groups, where not all the Euclidean structures are available and one has to find alternative ways. Moreover, as we mentioned, they show that some inequalities are independent on the dilation structure and on the norm one takes on  $\mathbb{R}^{n}$. To the best of our knowledge, all the known results rely on the fact that the appearing norm is the Euclidean one.
In the technical part, the present paper relies on the extensive use of the Euler operator and this is one of our main ideas, to use it in proving such inequalities beyond $\mathbb{R}^{n}$.

For the convenience of the reader let us summarise the obtained results.
Let $\mathbb{G}$ be a homogeneous group of homogeneous dimension $Q$ and let us fix any homogeneous quasi-norm $|\cdot|$. Then we prove the following results:
\begin{itemize}
\item
Let $f\in C_{0}^{\infty}(\mathbb{G}\backslash\{0\})$ be a real-valued function and let $1<p<\infty$. Then we have the following identity:
\begin{equation}\label{iLH2}
\left\|\frac{p}{Q}\mathbb{E} f\right\|^{p}_{L^{p}(\mathbb{G})}
-\|f\|^{p}_{L^{p}(\mathbb{G})}=
p\int_{\mathbb{G}}I_{p}\left(f,-\frac{p}{Q}\mathbb{E} f\right)\left|f+\frac{p}{Q}\mathbb{E} f\right|^{2}dx,
\end{equation}
where $I_{p}$ is given by
$$
I_{p}(h,g)=(p-1)\int_{0}^{1}|\xi h+(1-\xi)g|^{p-2}\xi d\xi
$$
and the Euler type operator
$$\mathbb{E} =|x|\frac{d}{d|x|}$$
is defined in \eqref{EQ:def-Euler}. It can be described by the property that if $f:\mathbb G\backslash \{0\}\to\mathbb R$ is differentiable, then
$ \mathbb{E}(f)=\nu f$
if and only if a function $f$ is positively homogeneous of order $\nu$.
\item
Using that $I_{p}\geq 0$, identity \eqref{iLH2} implies the generalised $L^{p}$-Sobolev
inequality
\begin{equation}\label{iLpSoboleveq}
\|f\|_{L^{p}(\mathbb{G})}\leq\frac{p}{Q}\left\|\mathbb{E} f\right\|_{L^{p}(\mathbb{G})},
\quad 1<p<\infty,
\end{equation}
for all real-valued functions $f\in C_{0}^{\infty}(\mathbb{G}\backslash\{0\}).$
Inequality \eqref{iLpSoboleveq} is also true for complex-valued functions. In the case $p=2$ and $Q\geq3$, the inequality \eqref{iLpSoboleveq} for any complex-valued $f\in C_{0}^{\infty}(\mathbb{G}\backslash\{0\})$ is equivalent to Hardy's inequality for any complex-valued $g\in C_{0}^{\infty}(\mathbb{G}\backslash\{0\})$:
\begin{equation}\label{Hardy}
\left\|\frac{g}{|x|}\right\|_{L^{2}(\mathbb{G})}\leq\frac{2}{Q-2}\left\|\mathcal{R} g\right\|_{L^{2}(\mathbb{G})}=\frac{2}{Q-2}\left\|\frac{1}{|x|}\mathbb{E} g\right\|_{L^{2}(\mathbb{G})},
\end{equation}
where $\mathbb{E}:=|x| \mathcal{R}$, and $\mathcal{R}$ is the radial derivative operator on $\mathbb{G}$ (see \eqref{dfdr} for the precise expression).
In the case $1<p<Q$, the inequality \eqref{iLpSoboleveq} for any complex-valued $f\in C_{0}^{\infty}(\mathbb{G}\backslash\{0\})$ implies Hardy's inequality:
\begin{equation}\label{Hardy_p}
\left\|\frac{f}{|x|}\right\|_{L^{p}(\mathbb{G})}\leq\frac{p}{Q-p}\left\|\mathcal{R} f\right\|_{L^{p}(\mathbb{G})}=\frac{p}{Q-p}\left\|\frac{1}{|x|}\mathbb{E} f\right\|_{L^{p}(\mathbb{G})}.
\end{equation}
\item
For any complex-valued function $f\in C^{\infty}_{0}(\mathbb{G}\backslash\{0\})$  and any $\alpha\in\mathbb{R}$ we have the following weighted identity:
\begin{multline}\label{iawS}
\qquad \left\|\frac{1}{|x|^{\alpha}}\mathbb{E} f\right\|^{2}_{L^{2}(\mathbb{G})}=
\left(\frac{Q}{2}-\alpha\right)^{2}\left\|\frac{f}{|x|^{\alpha}}
\right\|^{2}_{L^{2}(\mathbb{G})}\\
+\left\|\frac{1}{|x|^{\alpha}}\mathbb{E} f+\frac{Q-2\alpha}{2|x|^{\alpha}}f
\right\|^{2}_{L^{2}(\mathbb{G})}.
\end{multline}
Identity \eqref{iawS} implies several different estimates. For example, for
$\alpha=1$ we get the following generalised weighted Sobolev
inequality:
\begin{equation}\label{awSoboleveq}
\left\|\frac{f}{|x|}
\right\|_{L^{2}(\mathbb{G})}\leq
\frac{2}{Q-2}\left\|\frac{1}{|x|}\mathbb{E} f\right\|_{L^{2}(\mathbb{G})},\quad Q\geq 3,
\end{equation}
for all complex-valued functions $f\in C_{0}^{\infty}(\mathbb{G}\backslash\{0\})$.
For every $\alpha\in\mathbb{R}$ for which $Q-2\alpha\neq0$, since the last term in \eqref{iawS} in non-negative, we obtain
\begin{equation}\label{awSoboleveq-g0}
\qquad \left\|\frac{f}{|x|^{\alpha}}\right\|_{L^{2}(\mathbb{G})}\leq
\frac{2}{|Q-2\alpha|}
\left\|\frac{1}{|x|^{\alpha}}\mathbb{E} f\right\|_{L^{2}(\mathbb{G})},
\end{equation}
with sharp constant, see
Corollary \ref{waSobolev}.
\item
Let $1<p<\infty$, $k\in\mathbb{N}$ and $\alpha\in\mathbb{R}$ be such that $Q\neq \alpha p$. Then for any complex-valued function $f\in C^{\infty}_{0}(\mathbb{G}\backslash\{0\})$ we have weighted higher order $L^{p}$-Sobolev type inequality:
\begin{equation}
\left\|\frac{f}{|x|^{\alpha}}\right\|_{L^{p}(\mathbb{G})}\leq
\left|\frac{p}{Q-\alpha p}\right|^{k}\left\|\frac{1}{|x|^{\alpha}}\mathbb{E}^{k}f\right\|_{L^{p}(\mathbb{G})}.
\end{equation}
In the case $p=2$ an interesting feature is that we also obtain the exact formula for the remainder which yields the sharpness of the constants as well. For $\alpha\in\mathbb{R}$ and $Q-2\alpha\neq 0$ we have
the higher order Sobolev-Rellich type inequalities:
\begin{equation}\label{awSoboleveq-high-R2}
\left\|\frac{f}{|x|^{\alpha}}
\right\|_{L^{2}(\mathbb{G})}\leq
\left(\frac{2}{|Q-2\alpha|}\right)^{k}
\left\|\frac{1}{|x|^{\alpha}}\mathbb{E}^{k}f\right\|_{L^{2}(\mathbb{G})},
\end{equation}
for all complex-valued functions $f\in C_{0}^{\infty}(\mathbb{G}\backslash\{0\})$.
The constant in the right-hand side of \eqref{awSoboleveq-high-R2} is sharp, and is attained if and only if $f=0$.
Moreover, for all $k\in\mathbb N$ and $\alpha\in\mathbb{R}$ there is an explicit formula for the remainder:
\begin{multline}\label{equality-high-rem0}
\qquad \left\|\frac{1}{|x|^{\alpha}}\mathbb{E}^{k}f\right\|^{2}_{L^{2}(\mathbb{G})}=
\left(\frac{Q-2\alpha}{2}\right)^{2k}\left\|\frac{f}{|x|^{\alpha}}
\right\|^{2}_{L^{2}(\mathbb{G})}\\
+\sum_{m=1}^{k}\left(\frac{Q-2\alpha}{2}\right)^{2k-2m}
\left\|\frac{1}{|x|^{\alpha}}\mathbb{E}^{m}f+ \frac{Q-2\alpha}{2|x|^{\alpha}}\mathbb{E}^{m-1}f\right\|^{2}_{L^{2}(\mathbb{G})}.
\end{multline}
\item
The semi-normed space $(\mathfrak{L}^{k,p}(\mathbb{G}),\|\cdot\|_{\mathfrak{L}^{k,p}(\mathbb{G})})$, $k\in \mathbb{Z}$, equipped with a semi-norm $\|f\|_{\mathfrak{L}^{k,p}(\mathbb{G})}:=\|\mathbb{E}^{k} f\|_{L^{p}(\mathbb{G})}$ is a complete space. The norm of the embedding operator $\iota : (\mathfrak{L}^{k,p}(\mathbb{G}), \|\cdot\|_{\mathfrak{L}^{k,p}(\mathbb{G})})\hookrightarrow (L^{p}(\mathbb{G}), \|\cdot\|_{L^{p}(\mathbb{G})})$ for $k\in \mathbb{N}$ satisfies
\begin{equation}
\|\iota\|_{\mathfrak{L}^{k,p}(\mathbb{G})\rightarrow L^{p}(\mathbb{G})}\leq\left(\frac{p}{Q}\right)^k, \;1<p<\infty,
\end{equation}
where we understand the embedding $\iota$ as an embedding of semi-normed subspace of $L^{p}(\mathbb{G})$.

\item
Let $|\mathbb{E}|=(\mathbb{E}\mathbb{E}^{*})^{\frac{1}{2}}$. Then the semi-normed space $(\mathbb{H}^{\beta}(\mathbb{G}),\|\cdot\|_{\mathbb{H}^{\beta}(\mathbb{G})})$, $\beta\in \mathbb{C},$ equipped with a semi-norm $\|f\|_{\mathbb{H}^{\beta}(\mathbb{G})}:=\|\mathbb{|E|}^{\beta} f\|_{L^{2}(\mathbb{G})}$ is a complete space. The norm of the embedding operator $\iota : (\mathbb{H}^{\beta}(\mathbb{G}), \|\cdot\|_{\mathbb{H}^{\beta}(\mathbb{G})})\hookrightarrow (L^{2}(\mathbb{G}), \|\cdot\|_{L^{2}(\mathbb{G})})$ satisfies
\begin{equation}
\|\iota\|_{\mathbb{H}^{\beta}(\mathbb{G})\rightarrow L^{2}(\mathbb{G})}\leq C\left(k-\frac{\beta}{2},k\right)\left(\frac{2}{Q}\right)^{{\rm Re} \beta }, \; \beta\in \mathbb{C_{+}}, \;k>\frac{{\rm Re} \beta}{2}, \;k\in \mathbb{N},
\end{equation}
where we understand the embedding $\iota$ as an embedding of semi-normed subspace of $L^{2}(\mathbb{G})$, and with
$$C(\beta,k)=\frac{\Gamma(k+1)}{|\Gamma(\beta)\Gamma(k-\beta)|}\frac{2^{k-{\rm Re}\beta}}{{\rm Re}\beta (k-{\rm Re}\beta)}.
$$
\item
Let $\Omega$ be a bounded open subset of $\mathbb{G}$. If $1<p<\infty, \;f\in \widehat{\mathfrak{L}}_{0}^{1,p}(\Omega)$ and $\mathcal{R}f\equiv\frac{1}{|x|}\mathbb{E}f\in L^{p}(\Omega)$, then we have the following Poincar\'{e} type inequality on $\Omega\subset\mathbb{G}$:
\begin{equation}
\|f\|_{L^{p}(\Omega)}\leq  \frac{R p}{Q}\|\mathcal{R}f\|_{L^{p}(\Omega)}=\frac{R p}{Q}\left\|\frac{1}{|x|}\mathbb{E}f\right\|_{L^{p}(\Omega)},
\end{equation}
where $R=\underset{x\in \Omega}{\rm sup}|x|$, and $\widehat{\mathfrak{L}}_{0}^{1,p}(\Omega)$ is the completion of $C^{\infty}_{0}(\Omega\backslash\{0\})$ with respect to the norm
$$\|f\|_{\widehat{\mathfrak{L}}^{1,p}(\Omega)}:=\|f\|_{L^{p}(\Omega)}+\|\mathbb{E}f\|_{L^{p}(\Omega)}, \;\;1<p<\infty.$$
\item
Let $1<\gamma<\infty$ and $\max\{1,\gamma-1\}<q<\infty$. Then we have the continuous embedding
$$W^{1}_{0}L_{Q,q,\frac{q-1}{q},
\frac{q-\gamma}{q}}(\mathbb{G})\hookrightarrow \mathfrak{L}_{\infty,q,-\frac{1}{q},
-\frac{\gamma}{q}}(\mathbb{G}),$$
where $W^{1}_{0}L_{Q,q,\frac{q-1}{q},\frac{q-\gamma}{q}}(\mathbb{G})$ and $\mathfrak{L}_{\infty,q,-\frac{1}{q},-\frac{\gamma}{q}}(\mathbb{G})$ are defined in \eqref{SLZdef1} and \eqref{LZdef}, respectively. In particular, for any $R>0$, the inequality
\begin{multline}\label{SLZ111}
\left(\int_{\mathbb{G}}\frac{\chi_{B(0,eR)}(x)|f-f_{R}|^{q}+
\chi_{B^{c}(0,eR)}(x)|f-f_{e^{2}R}|^{q}}{
\left|\log\left|\log\frac{e R}{|x|}\right|\right|^{\gamma}\left|\log\frac{e R}{|x|}\right|}\frac{dx}{|x|^{Q}}\right)^{\frac{1}{q}}\\
\leq \frac{q}{\gamma-1}\left(\int_{\mathbb{G}}|x|^{q-Q}
\left|\log\frac{e R}{|x|}\right|^{q-1}\left|\log\left|\log\frac{e R}{|x|}\right|\right|^{q-\gamma}\left|\frac{1}{|x|}\mathbb{E}f\right|
^{q}dx\right)^{\frac{1}{q}}
\end{multline}
holds for all $f\in W^{1}_{0}L_{{Q,q,\frac{q-1}{q},\frac{q-\gamma}{q}}}(\mathbb{G})$, where the embedding constant $\frac{q}{\gamma-1}$ is sharp and $f_{R}=f(R\frac{x}{|x|})$.
\end{itemize}

An interesting observation is that the constants in most of the obtained inequalities are sharp for any quasi-norm $|\cdot|$, that is, they do not depend on the quasi-norm $|\cdot|$. Therefore, these inequalities are new already in the Euclidean setting of $\mathbb{R}^{n}$. Moreover, one of the novelty of this paper is the analysis of the fractional powers of the Euler operators in $\mathbb{R}^{n}$, as well as on the homogeneous groups. We refer to  \cite{RTY17} and \cite{RY17} for the applications of such inequalities to nonlinear Schr\"{o}dinger type equations on homogeneous groups, namely on graded groups. In particular, the authors expressed the best constants of the Sobolev and Gagliardo-Nirenberg inequalities in the variational form as well as in terms of the ground state solutions of the corresponding nonlinear subelliptic equations.

In Section \ref{SEC:prelim} we briefly recall the main concepts of homogeneous groups and fix the notation. In Section \ref{Sec3} we derive versions of  Sobolev type inequalities on homogeneous groups and discuss their consequences including higher order Sobolev-Rellich inequalities. Euler-Hilbert-Sobolev and Euler-Sobolev spaces on homogeneous groups are defined in Section \ref{SEC:EHS}. In Section \ref{SEC:Poincare} we consider an analogue of Poincar\'{e} inequality on homogeneous groups. Finally, in Section \ref{SLZ} we discuss an analogue of the critical Hardy inequality and Sobolev-Lorentz-Zygmund spaces on homogeneous groups.

\section{Euler operator on homogeneous groups}
\label{SEC:prelim}

In this section we very briefly recall the necessary notions and fix the notation for homogeneous groups. We also describe the Euler operator that will play a crucial role in our analysis.

We work with a family of dilations of a Lie algebra $\mathfrak{g}$, which
is a family of linear mappings of the following form
$$D_{\lambda}={\rm Exp}(A \,{\rm ln}\lambda)=\sum_{k=0}^{\infty}
\frac{1}{k!}({\rm ln}(\lambda) A)^{k},$$
where $A$ is a diagonalisable linear operator on the Lie algebra $\mathfrak{g}$
with positive eigenvalues,
and each $D_{\lambda}$ is a morphism of $\mathfrak{g}$,
that is, a linear mapping
from $\mathfrak{g}$ to itself satisfying:
$$\forall X,Y\in \mathfrak{g},\, \lambda>0,\;
[D_{\lambda}X, D_{\lambda}Y]=D_{\lambda}[X,Y],$$
where $[X,Y]:=XY-YX$ is the Lie bracket. Then, a {\em homogeneous group} is a connected simply connected Lie group whose Lie algebra is equipped with dilations. It induces the dilation structure on the homogeneous group $\mathbb G$ which we denote by $D_{\lambda}x$ or just by $\lambda x$.

Let $dx$ be the Haar measure on $\mathbb{G}$ and let $|S|$ denote the volume of a measurable set $S\subset \mathbb{G}$.
Then we have
\begin{equation}
|D_{\lambda}(S)|=\lambda^{Q}|S| \quad {\rm and}\quad \int_{\mathbb{G}}f(\lambda x)
dx=\lambda^{-Q}\int_{\mathbb{G}}f(x)dx,
\end{equation}
where $Q$ is the homogeneous dimension of $\mathbb G$:
$$Q := {\rm Tr}\,A.$$
A {\em homogeneous quasi-norm} on $\mathbb G$ is
a continuous non-negative function
$$\mathbb{G}\ni x\mapsto |x|\in [0,\infty)$$
satisfying the following properties
\begin{itemize}
\item   $|x^{-1}| = |x|$ for all $x\in \mathbb{G}$,
\item  $|\lambda x|=\lambda |x|$ for all
$x\in \mathbb{G}$ and $\lambda >0$,
\item  $|x|= 0$ if and only if $x=0$.
\end{itemize}

The quasi-ball centred at $x\in\mathbb{G}$ with radius $R > 0$ can be defined by
$$B(x,R):=\{y\in \mathbb{G}: |x^{-1}y|<R\}.$$
We also use the notation
$$B^{c}(x,R):=\{y\in \mathbb{G}: |x^{-1}y|\geq R\}.$$
The polar decomposition on homogeneous groups will be very useful for our analysis, and it can be formulated as follows:
there is a (unique)
positive Borel measure $\sigma$ on the
unit sphere
\begin{equation}\label{EQ:sphere}
\wp:=\{x\in \mathbb{G}:\,|x|=1\},
\end{equation}
such that for all $f\in L^{1}(\mathbb{G})$ we have
\begin{equation}\label{EQ:polar}
\int_{\mathbb{G}}f(x)dx=\int_{0}^{\infty}
\int_{\wp}f(ry)r^{Q-1}d\sigma(y)dr.
\end{equation}
We refer to Folland and Stein \cite{FS-Hardy} for its proof (see also \cite[Section 3.1.7]{FR} for a detailed discussion).

From now on, we fix a basis $\{X_{1},\ldots,X_{n}\}$ of $\mathfrak{g}$
such that
$$AX_{k}=\nu_{k}X_{k}$$
for every $k$, so that the matrix $A$ can be taken to be
$A={\rm diag} (\nu_{1},\ldots,\nu_{n}).$
Then each $X_{k}$ is homogeneous of degree $\nu_{k}$ and
$$
Q=\nu_{1}+\cdots+\nu_{n}.
$$
The decomposition of ${\exp}_{\mathbb{G}}^{-1}(x)$ in the Lie algebra $\mathfrak g$ defines the vector
$$e(x)=(e_{1}(x),\ldots,e_{n}(x))$$
by the formula
$${\exp}_{\mathbb{G}}^{-1}(x)=e(x)\cdot \nabla\equiv\sum_{j=1}^{n}e_{j}(x)X_{j},$$
where $\nabla=(X_{1},\ldots,X_{n})$.
On the other hand, we have the equality
$$x={\exp}_{\mathbb{G}}\left(e_{1}(x)X_{1}+\ldots+e_{n}(x)X_{n}\right).$$
Taking into account the homogeneity and denoting $x=ry,\,y\in \wp,$ this implies
$$
e(x)=e(ry)=(r^{\nu_{1}}e_{1}(y),\ldots,r^{\nu_{n}}e_{n}(y)).
$$
So one has
\begin{equation*}
\frac{d}{d|x|}f(x)=\frac{d}{dr}f(ry)=
 \frac{d}{dr}f({\exp}_{\mathbb{G}}
\left(r^{\nu_{1}}e_{1}(y)X_{1}+\ldots
+r^{\nu_{n}}e_{n}(y)X_{n}\right)).
\end{equation*}
Denoting by
\begin{equation}\label{EQ:Euler}
\mathcal{R} :=\frac{d}{dr},
\end{equation}
for all $x\in \mathbb G$ this gives the equality
\begin{equation}\label{dfdr}
	\frac{d}{d|x|}f(x)=\mathcal{R}f(x),
\end{equation}
for each homogeneous quasi-norm $|x|$ on a homogeneous group $\mathbb G.$

Let us state a relation between $\mathcal{R}$ and Euler's operator:
\begin{lem}\label{L:Euler}
Define the Euler operator by
\begin{equation}\label{EQ:def-Euler}
\mathbb{E}:=|x| \mathcal{R}.
\end{equation}
If $f:\mathbb G\backslash \{0\}\to\mathbb R$ is differentiable, then
$$
\mathbb{E}(f)=\nu f
 \; \textrm{ if and only if }\;
 f(r x)=r^{\nu} f(x)\;\; (\forall r>0, \nu\in\mathbb{R}, x\not=0).$$
\end{lem}
\begin{proof}[Proof of Lemma \ref{L:Euler}]
If a function $f$ is positively homogeneous of order $\nu$, that is, if $f(r x)=r^{\nu}f(x)$ holds for all $r>0$ and $x:=\rho y\not=0, \,\,y\in \wp,$ then using \eqref{dfdr} for such $f$, it follows that
$$
\mathbb{E}f=\nu f(x).
$$
Conversely, let us fix $x\not=0$ and define $g(r):=f(rx)$.
Using \eqref{dfdr}, the equality $\mathbb{E}f(rx)=\nu f(rx)$ means that
$$
g'(r)=\frac{d}{dr}f(rx)=\frac{1}{r} \mathbb{E}f(rx)=\frac{\nu}{r} f(rx)=\frac{\nu}{r}g(r).
$$
Consequently, $g(r)=g(1) r^{\nu}$, i.e. $f(rx)=r^{\nu} f(x)$ and thus $f$ is positively homogeneous of order $\nu$.
\end{proof}

We now establish some properties of the Euler operator $\mathbb{E}$.

\begin{lem}\label{L2:E} We have
\begin{equation}\label{L2:E1}
\mathbb{E^{*}}=-Q\mathbb{I}-\mathbb{E}
\end{equation}
where $\mathbb{I}$ and $\mathbb{E^{*}}$ are the identity operator and the formal adjoint operator of $\mathbb{E}$, respectively.
\end{lem}
\begin{proof}[Proof of Lemma \ref{L2:E}]
Let us calculate a formal adjoint operator of $\mathbb{E}$ on $C_{0}^{\infty}(\mathbb{G}\backslash\{0\})$. We have
\begin{align*}\int_{\mathbb{G}}\mathbb{E}f(x)\overline{g(x)}dx & =\int_{0}^{\infty}\int_{\wp}
\frac{d}{dr}f(ry)\overline{g(ry)} r^{Q}d\sigma(y)dr\\ & =-\int_{0}^{\infty}\int_{\wp}
f(ry)\left(Qr^{Q-1}\overline{g(ry)}+r^{Q}\frac{d}{dr}\overline{g(ry)}\right)d\sigma(y)dr
\\ & =
-\int_{\mathbb{G}}f(x)(Q+\mathbb{E})\overline{g(x)}dx,
\end{align*}
by integration by parts using formula \eqref{dfdr}.
\end{proof}
We now show that the operator $\mathbb{A}=\mathbb{E}\mathbb{E}^{*}$ is Komatsu-non-negative (see e.g. \cite{Kom66} or \cite{Martinez-Sanz:fractional}) in $L^{2}(\mathbb{G})$, that is, $(-\infty,0)$ is included in the resolvent set $\rho(\mathbb{A})$ of $\mathbb{A}$ and
\begin{equation*}
\exists M>0, \; \forall\lambda>0, \quad \|(\lambda+\mathbb{A})^{-1}\|_{L^{2}(\mathbb{G})\rightarrow L^{2}(\mathbb{G})}\leq M\lambda^{-1}.
\end{equation*}
\begin{lem}\label{L2:Komatsu}
The operator $\mathbb{A}=\mathbb{E}\mathbb{E}^{*}$ is Komatsu-non-negative in $L^{2}(\mathbb{G})$:
\begin{equation}\label{Komatsu0}
\quad \|(\lambda+\mathbb{A})^{-1}\|_{L^{2}(\mathbb{G})\rightarrow L^{2}(\mathbb{G})}\leq \lambda^{-1}, \forall\lambda>0.
\end{equation}
\end{lem}
\begin{proof}[Proof of Lemma \ref{L2:Komatsu}] We start with $f\in C_{0}^{\infty}(\mathbb{G}\backslash\{0\})$. A direct calculation shows
\begin{multline}\label{Komatsu1}
\|(\lambda \mathbb{I}+\mathbb{A})f\|^{2}_{L^{2}(\mathbb{G})}=\|(\lambda \mathbb{I}-\mathbb{E}(Q\mathbb{I}+\mathbb{E}))f\|^{2}_{L^{2}(\mathbb{G})} \\ =\lambda^{2}\left\|f\right\|^{2}_{L^{2}(\mathbb{G})}+
\left\|\mathbb{E}(Q\mathbb{I}+\mathbb{E})f\right\|^{2}_{L^{2}(\mathbb{G})}-2\lambda {\rm Re} \int_{\mathbb{G}}f(x)
\overline{Q\mathbb{E}f+\mathbb{E}^{2}f}dx.
\end{multline}
Since
\begin{align*}
{\rm Re}\int_{G}f(x)\overline{\mathbb{E}^{2}f}dx & ={\rm Re}\int_{0}^{\infty}\int_{\wp}
f(ry)\overline{\frac{d}{dr}(\mathbb{E}f(ry))} r^{Q} d\sigma(y)dr\\
& =-{\rm Re}\int_{0}^{\infty}\int_{\wp}
\overline{(\mathbb{E}f(ry))} \left(r^{Q}\frac{d}{dr}f(ry)+Qr^{Q-1}f(ry)\right)d\sigma(y)dr\\
 & =-\left\|\mathbb{E}f\right\|^{2}_{L^{2}(\mathbb{G})}-Q{\rm Re} \int_{\mathbb{G}}\overline{\mathbb{E}f(x)}f(x)dx,
\end{align*}
we have
$$-2\lambda {\rm Re} \int_{\mathbb{G}}f(x)\overline{Q\mathbb{E}f(x)+\mathbb{E}^{2}f(x)}dx$$
\begin{equation}\label{Komatsu2}
=-2\lambda Q {\rm Re} \int_{\mathbb{G}}f(x)
\overline{\mathbb{E}f(x)}dx-2\lambda {\rm Re} \int_{\mathbb{G}}f(x)
\overline{\mathbb{E}^{2}f(x)}dx=2\lambda\left\|\mathbb{E}f\right\|^{2}_{L^{2}(\mathbb{G})}.
\end{equation}
Combining \eqref{Komatsu1} with \eqref{Komatsu2}, we obtain
$$\|(\lambda \mathbb{I}-\mathbb{E}(Q\mathbb{I}+\mathbb{E}))f\|^{2}_{L^{2}(\mathbb{G})}=\lambda^{2}\left\|f\right\|^{2}_{L^{2}(\mathbb{G})}+
2\lambda\left\|\mathbb{E}f\right\|^{2}_{L^{2}(\mathbb{G})}+\left\|\mathbb{E}(Q\mathbb{I}+\mathbb{E})f\right\|^{2}_{L^{2}(\mathbb{G})}.$$
By dropping positive terms, it follows that
$$\|(\lambda \mathbb{I}-\mathbb{E}(Q\mathbb{I}+\mathbb{E}))f\|^{2}_{L^{2}(\mathbb{G})}\geq\lambda^{2}\left\|f\right\|^{2}_{L^{2}(\mathbb{G})},$$
which implies \eqref{Komatsu0}.
\end{proof}
\begin{lem}\label{L2:E2} For all complex-valued functions $f\in C_{0}^{\infty}(\mathbb{G}\backslash\{0\})$ we have
\begin{equation}\label{L2:E3}
\left\|\mathbb{E}f\right\|_{L^{2}(\mathbb{G})}=\left\|\mathbb{E^{*}}f\right\|_{L^{2}(\mathbb{G})}.
\end{equation}
\end{lem}
\begin{rem} The following identity is easy to check from the definition:
$$\mathbb{A}=\mathbb{E}\mathbb{E^{*}}=\mathbb{E^{*}}\mathbb{E}.$$
\end{rem}
\begin{proof}[Proof of Lemma \ref{L2:E2}] Using the representation of $\mathbb{E^{*}}$ in \eqref{L2:E1}, we get
$$\left\|\mathbb{E^{*}}f\right\|^{2}_{L^{2}(\mathbb{G})}=\left\|(-Q\mathbb{I}-\mathbb{E})f\right\|^{2}_{L^{2}(\mathbb{G})}$$
\begin{equation}\label{L2:E4}
=Q^2\left\|f\right\|^{2}_{L^{2}(\mathbb{G})}+2Q{\rm Re}\int_{\mathbb{G}}f(x)\overline{\mathbb{E}f(x)}dx+\left\|\mathbb{E}f\right\|^{2}_{L^{2}(\mathbb{G})}.
\end{equation}
Then we have
\begin{align*} 2Q{\rm Re}\int_{\mathbb{G}}f(x)\overline{\mathbb{E}f(x)}dx & =2Q{\rm Re}\int_{0}^{\infty}\int_{\wp}
f(ry)\frac{d}{dr}\overline{f(ry)} r^{Q} d\sigma(y)dr
\\ & =Q\int_{0}^{\infty}r^{Q}\int_{\wp}
\frac{d}{dr}(|f(ry)|^2) d\sigma(y)dr
\\ & =-Q^{2}\int_{0}^{\infty}\int_{\wp}
|f(ry)|^2 r^{Q-1} d\sigma(y)dr
\end{align*}
\begin{equation}\label{L2:E5}=-Q^{2}\left\|f\right\|^{2}_{L^{2}(\mathbb{G})}.
\end{equation}
Combining this with \eqref{L2:E4} we obtain \eqref{L2:E3}.
\end{proof}
In \eqref{L2:E3} replacing $f$ by $\mathbb{E}f$, we get
\begin{cor}\label{AE} For all complex-valued functions $f\in C_{0}^{\infty}(\mathbb{G}\backslash\{0\})$ we have
\begin{equation}\label{A0}
\left\|\mathbb{A}f\right\|_{L^{2}(\mathbb{G})}=\left\|\mathbb{E}^{2}f\right\|_{L^{2}(\mathbb{G})}.
\end{equation}
\end{cor}

\section{Sobolev type inequalities}
\label{Sec3}
In this section and in the sequel we adopt all the notation introduced in Section \ref{SEC:prelim} concerning homogeneous groups and the operator $\mathbb{E}$.
\subsection{Sobolev type inequalities} If $1<p,p^{*}<\infty$ and
\begin{equation}
\frac{1}{p}=\frac{1}{p^{*}}-\frac{1}{n},
\end{equation}
then the (Euclidean) Sobolev inequality has the form
\begin{equation}
\|g\|_{L^{p}(\mathbb{R}^{n})}\leq C(p)\|\nabla g\|_{L^{p^{*}}(\mathbb{R}^{n})},
\end{equation}
where $\nabla$ is the standard gradient in $\mathbb{R}^{n}$.

In \cite{OS09} the following Sobolev type inequality with respect to the operator $x\cdot \nabla$ instead of $\nabla$ has been considered:
\begin{equation}\label{Sob}
\|g\|_{L^{p}(\mathbb{R}^{n})}\leq C'(p)\|x\cdot\nabla g\|_{L^{q}(\mathbb{R}^{n})}.
\end{equation}
For any $\lambda>0$, by substituting $g(x)=h(\lambda x)$ into \eqref{Sob}, one readily observes that $p=q$ is a necessary condition to obtain \eqref{Sob}.

Let us now show the $L^{p}$-Sobolev type inequality and its remainder formula on the homogeneous group $\mathbb{G}$.

\begin{thm}\label{RemSobolev} Let $\mathbb{G}$ be a homogeneous group of homogeneous dimension $Q$ and
let $|\cdot|$ be any homogeneous quasi-norm on $\mathbb{G}$.
\begin{itemize}
\item[(i)] Then we have
\begin{equation}\label{LpSobolevinC}
\left\|f\right\|_{L^{p}(\mathbb{G})}\leq\frac{p}{Q}\left\|\mathbb{E} f\right\|_{L^{p}(\mathbb{G})}, \quad 1<p<\infty,
\end{equation}
for all complex-valued functions $f\in C_{0}^{\infty}(\mathbb{G}\backslash\{0\})$, where the constant $\frac{p}{Q}$ is sharp and the equality is attained if and only if $f=0$.

\item[(ii)] Denoting by
$$u:=u(x)=-\frac{p}{Q}\mathbb{E} f(x), \;v:=v(x)=f(x),$$
we have the identity
\begin{equation}\label{LH2}
\qquad
\left\|u\right\|^{p}_{L^{p}(\mathbb{G})}
-\left\|v\right\|^{p}_{L^{p}(\mathbb{G})}=p\int_{\mathbb{G}}I_{p}(v,u)|v-u|^{2}dx, \quad 1<p<\infty,
\end{equation}
for every real-valued functions $f\in C_{0}^{\infty}(\mathbb{G}\backslash\{0\})$, where
$$
I_{p}(h,g)=(p-1)\int_{0}^{1}|\xi h+(1-\xi)g|^{p-2}\xi d\xi.
$$
\item[(iii)]
In the case $p=2$, the identity \eqref{LH2} holds for all complex-valued functions $f\in C_{0}^{\infty}(\mathbb{G}\backslash\{0\})$ and has the following form
\begin{equation}\label{EQ:expL2}
\left\|\mathbb{E} f\right\|^{2}_{L^{2}(\mathbb{G})}=
\left(\frac{Q}{2}\right)^{2}\left\|f\right\|^{2}_{L^{2}(\mathbb{G})}+
\left\|\mathbb{E} f+\frac{Q}{2}f\right\|^{2}_{L^{2}(\mathbb{G})}.
\end{equation}
\item[(iv)] In the case $p=2$ and $Q\geq3$ the inequality \eqref{LpSobolevinC} is equivalent to Hardy's inequality for any $g\in C_{0}^{\infty}(\mathbb{G}\backslash\{0\})$:
\begin{equation}\label{Hardy}
\left\|\frac{g}{|x|}\right\|_{L^{2}(\mathbb{G})}\leq\frac{2}{Q-2}\left\|\mathcal{R} g\right\|_{L^{2}(\mathbb{G})}=\frac{2}{Q-2}\left\|\frac{1}{|x|}\mathbb{E} g\right\|_{L^{2}(\mathbb{G})}.
\end{equation}
\item[(v)] In the case $1<p<Q$ the inequality \eqref{LpSobolevinC} yields Hardy's inequality for any $f\in C_{0}^{\infty}(\mathbb{G}\backslash\{0\})$:
\begin{equation}\label{Hardy_p}
\left\|\frac{f}{|x|}\right\|_{L^{p}(\mathbb{G})}\leq\frac{p}{Q-p}\left\|\mathcal{R} f\right\|_{L^{p}(\mathbb{G})}.
\end{equation}
\end{itemize}
\end{thm}
\begin{rem} Let us consider the following Sobolev type inequality for all $1<p,q<\infty$:
\begin{equation}\label{L_pq}
\|g\|_{L^{p}(\mathbb{G})}\leq C(p)\|\mathbb{E}g\|_{L^{q}(\mathbb{G})}.
\end{equation}
For any $\lambda>0$, substituting $g(x)=h(\lambda x)$ into \eqref{L_pq} and using the fact the Euler operator is a homogeneous operator of order zero, we obtain that $p=q$ is a necessary condition to obtain \eqref{L_pq}.
\end{rem}
\begin{rem}\label{LpHardy}
Let us show that Part (ii) implies Part (i).
By dropping nonnegative term in the right-hand side of \eqref{LH2}, we obtain
\begin{equation}\label{LpHardyeq}
\left\|f\right\|_{L^{p}(\mathbb{G})}\leq\frac{p}{Q}\left\|\mathbb{E} f\right\|_{L^{p}(\mathbb{G})},\quad 1<p<\infty,
\end{equation}
for all real-valued functions $f\in C_{0}^{\infty}(\mathbb{G}\backslash\{0\}).$
Consequently, this inequality holds for all complex-valued functions by using the identity (cf. Davies \cite[p. 176]{Davies-bk:Semigroups-1980})
\begin{equation}\label{EQ:Davies-rc}
\forall z\in\mathbb C:\;
|z|^{p}=\left(\int_{-\pi}^{\pi}|\cos\theta|^{p} d\theta\right)^{-1}
\int_{-\pi}^{\pi}\left| {\rm Re}(z)\cos\theta+{\rm Im}(z)\sin\theta\right|^{p}d\theta,
\end{equation}
which implies from the representation $z=r(\cos\phi+i\sin\phi)$ by some manipulations.

So, the inequality \eqref{LpSobolevinC} is valid with the sharp constant $\frac{p}{Q}$. We now claim that this constant is attained only for $f=0$. By virtue of \eqref{EQ:Davies-rc}, it is enough to check it only for real-valued functions $f$. If the right-hand side of \eqref{LH2} is zero, then we must obtain
$$
-\frac{p}{Q}\mathbb {E} f(x)=f(x),
$$
which implies that $\mathbb{E}(f)=-\frac{Q}{p}f$. Taking into account the property of the Euler operator in Lemma \ref{L:Euler} it means that $f$ is positively homogeneous of order $-\frac{Q}{p}$, that is, there exists a function $h:\wp\to\mathbb C$ such that
\begin{equation}\label{EQ:homr1}
f(x)=|x|^{-\frac{Q}{p}}h\left(\frac{x}{|x|}\right),
\end{equation}
where $\wp$ is defined by \eqref{EQ:sphere}.
In particular, \eqref{EQ:homr1} means that $f$ can not be compactly supported unless it is zero.
\end{rem}

So, Remark \ref{LpHardy} shows that Part (ii), namely \eqref{LH2} implies Part (i) of Theorem \ref{RemSobolev}. Thus, for the part on Sobolev type inequality, we only need to prove Part (ii). Nevertheless, we also give an independent proof of \eqref{LpSobolevinC} for complex-valued functions without using the formula \eqref{EQ:Davies-rc}, especially since this calculation will be also useful in proving Part (ii) of Theorem \ref{RemSobolev}.

\begin{proof}[Proof of Theorem \ref{RemSobolev}]

Introducing polar coordinates $(r,y)=(|x|, \frac{x}{\mid x\mid})\in (0,\infty)\times\wp$ on $\mathbb{G}$, where the quasi-sphere $\wp$ is defined in \eqref{EQ:sphere}, and applying \eqref{EQ:polar} and integrating by parts, we get
$$
\int_{\mathbb{G}}
|f(x)|^{p}dx
=\int_{0}^{\infty}\int_{\wp}
|f(ry)|^{p} r^{Q-1}d\sigma(y)dr
$$
$$
=-\frac{p}{Q}\int_{0}^{\infty} r^{Q} \,{\rm Re}\int_{\wp}
|f(ry)|^{p-2} f(ry) \overline{\frac{df(ry)}{dr}}d\sigma(y)dr
$$
\begin{equation}\label{EQ:formula1}
=-\frac{p}{Q} {\rm Re}\int_{\mathbb{G}}
|f(x)|^{p-2}f(x)
\overline{\mathbb{E}f(x)}dx.
\end{equation}
Here using the H\"older's inequality for $\frac{1}{p}+\frac{1}{q}=1$ one calculates
$$\int_{\mathbb{G}}
|f(x)|^{p}dx
=-\frac{p}{Q} {\rm Re}\int_{\mathbb{G}}
|f(x)|^{p-2}f(x)
\overline{\mathbb{E}f(x)}dx
$$
$$
\leq \frac{p}{Q}\left(\int_{\mathbb{G}}
\left||f(x)|^{p-2}f(x)\right|^{q}dx\right)^{\frac{1}{q}}
\left(\int_{\mathbb{G}}
\left|\mathbb{E}f(x)\right|^{p}dx\right)^{\frac{1}{p}}
$$
$$
=\frac{p}{Q} \left(\int_{\mathbb{G}}
|f(x)|^{p}dx\right)^{1-\frac{1}{p}}
\left\|\mathbb{E}f(x)\right\|_{L^{p}(\mathbb{G})},
$$
which gives inequality \eqref{LpSobolevinC} in Part (i).

\medskip
We now prove Part (ii). Applying notations
$$u:=u(x)=-\frac{p}{Q}\mathbb{E} f,\;v:=v(x)=f(x),$$
formula \eqref{EQ:formula1}
can be rewritten as
\begin{equation}\label{vu2}
\|v\|_{L^{p}(\mathbb{G})}^{p}={\rm Re}\int_{\mathbb{G}}|v|^{p-2}v \overline{u} dx.
\end{equation}
For any real-valued functions $f$ formula \eqref{EQ:formula1} becomes
$$
\int_{\mathbb{G}}
|f(x)|^{p}dx
=-\frac{p}{Q} \int_{\mathbb{G}}
|f(x)|^{p-2}f(x)\mathbb{E}f(x)dx
$$
and \eqref{vu2} takes the form
\begin{equation}\label{vu}
\|v\|_{L^{p}(\mathbb{G})}^{p}=\int_{\mathbb{G}}|v|^{p-2}v u dx.
\end{equation}
On the other hand, for all $L^{p}$-integrable real-valued functions $u$ and $v$ we have
\begin{multline*}
\|u\|_{L^{p}(\mathbb{G})}^{p}-\|v\|_{L^{p}(\mathbb{G})}^{p}
+p\int_{\mathbb{G}}(|v|^{p}-|v|^{p-2}v u)dx \\
=\int_{\mathbb{G}}(|u|^{p}+(p-1)|v|^{p}-p|v|^{p-2}vu)dx\\
=p\int_{\mathbb{G}}I_{p}(v,u)|v-u|^{2}dx, \;1<p<\infty,
\end{multline*}
where
$$
I_{p}(v,u)=(p-1)\int_{0}^{1}|\xi v+(1-\xi)u|^{p-2}\xi d\xi.
$$
Combining this with \eqref{vu} one obtains
$$\left\|u\right\|^{p}_{L^{p}(\mathbb{G})}
-\left\|v\right\|^{p}_{L^{p}(\mathbb{G})}=p\int_{\mathbb{G}}I_{p}(v,u)|v-u|^{2}dx.
$$
The equality \eqref{LH2} is proved.

\medskip
Now we prove Part (iii). If $p=2$, the identity \eqref{vu2} becomes
\begin{equation}\label{vu3}
\|v\|_{L^{2}(\mathbb{G})}^{2}={\rm Re}\int_{\mathbb{G}}v \overline{u} dx.
\end{equation}
Then we get
\begin{align*}
\|u\|_{L^{2}(\mathbb{G})}^{2}-\|v\|_{L^{2}(\mathbb{G})}^{2} & =
\|u\|_{L^{2}(\mathbb{G})}^{2}-\|v\|_{L^{2}(\mathbb{G})}^{2}+
2\int_{\mathbb{G}} (|v|^{2}-{\rm Re}\, v \overline{u}) dx \\
& =
\int_{\mathbb{G}} (|u|^{2}+|v|^{2}-2{\rm Re}\, v \overline{u}) dx
\\ & =\int_{\mathbb{G}} |u-v|^{2} dx,
\end{align*}
which implies \eqref{EQ:expL2}.

Now let us show Part (iv). First we show that the inequality \eqref{LpSobolevinC} implies \eqref{Hardy}. Let $g=|x|f$. Then we have
\begin{align*}\|\mathbb{E}f(x)\|_{L^{2}(\mathbb{G})}^{2}=\left\|\mathbb{E}
	\frac{g}{|x|}\right\|_{L^{2}(\mathbb{G})}^{2}
&=\int_{0}^{\infty}\int_{\wp}
\left|\left(r\frac{d}{dr}\right)\frac{g(ry)}{r}\right|^{2}r^{Q-1}d\sigma(y)dr
\\&=\left\|-\frac{g}{|x|}+\frac{d}{d|x|}g\right\|_{L^{2}(\mathbb{G})}^{2}
\end{align*}
\begin{equation}\label{Eq1}
=\left\|\frac{g}{|x|}\right\|_{L^{2}(\mathbb{G})}^{2}-2 {\rm Re} \int_{\mathbb{G}}\frac{g}{|x|} \overline{\frac{d}{d|x|}g}dx+\left\|\frac{d}{d|x|}g\right\|_{L^{2}(\mathbb{G})}^{2}.
\end{equation}
Since
\begin{equation*}
-2 {\rm Re} \int_{\mathbb{G}}\frac{g}{|x|} \overline{\frac{d}{d|x|}g}dx=
-2 {\rm Re}\int_{0}^{\infty}\int_{\wp}
\frac{g(ry)}{r} \overline{\frac{d}{dr}g(ry)}r^{Q-1}d\sigma(y)dr
\end{equation*}
$$=-{\rm Re} \int_{0}^{\infty}\int_{\wp}
\frac{d}{dr}(|\overline{g}|^{2})r^{Q-2}d\sigma(y)dr=(Q-2){\rm Re} \int_{0}^{\infty}\int_{\wp}|\overline{g}|^{2}r^{Q-3}d\sigma(y)dr$$
$$=(Q-2)\left\|\frac{g}{|x|}\right\|_{L^{2}(\mathbb{G})}^{2},$$
it follows that
\begin{equation}\label{Eq01}
\|\mathbb{E}f(x)\|_{L^{2}(\mathbb{G})}^{2}=
(Q-1)\left\|\frac{g}{|x|}\right\|_{L^{2}(\mathbb{G})}^{2}
+\left\|\frac{d}{d|x|}g\right\|_{L^{2}(\mathbb{G})}^{2}.
\end{equation}
We see from \eqref{LpSobolevinC} and \eqref{Eq01} that
$$\left\|\frac{g}{|x|}\right\|_{L^{2}(\mathbb{G})}^{2}\leq\frac{4}{Q^2}
\left((Q-1)\left\|\frac{g}{|x|}\right\|_{L^{2}(\mathbb{G})}^{2}+
\left\|\frac{d}{d|x|}g\right\|_{L^{2}(\mathbb{G})}^{2}\right),$$
which gives \eqref{Hardy}.

Conversely, we assume that \eqref{Hardy} holds. For $f=g/|x|$ we have
$$\left\|\frac{d}{d|x|}\left(|x|f\right)\right\|_{L^{2}(\mathbb{G})}^{2}=
\left\|f+\mathbb{E}f\right\|_{L^{2}(\mathbb{G})}^{2}
=\left\|f\right\|_{L^{2}(\mathbb{G})}^{2}$$
$$+2 {\rm Re} \int_{\mathbb{G}}f(x) \overline{\mathbb{E}f(x)}dx+\|\mathbb{E}f\|_{L^{2}(\mathbb{G})}^{2}.$$
Since by \eqref{L2:E5}, it follows from \eqref{Hardy} that
$$\left\|f\right\|_{L^{2}(\mathbb{G})}^{2}\leq\frac{4}{(Q-2)^2}\left(\|\mathbb{E}f\|_{L^{2}(\mathbb{G})}^{2}-(Q-1)\left\|f\right\|_{L^{2}(\mathbb{G})}^{2}\right),$$
which implies
$$\left\|f\right\|_{L^{2}(\mathbb{G})}\leq\frac{2}{Q}\|\mathbb{E}
f\|_{L^{2}(\mathbb{G})}.$$

Let us now prove Part (v). We will prove that the inequality \eqref{LpSobolevinC} gives \eqref{Hardy_p}. We have
$$\|\mathcal{R}(|x|f)\|_{L^{p}(\mathbb{G})}=\|\mathbb{E}f+f\|_{L^{p}(\mathbb{G})}\geq
\|\mathbb{E}f\|_{L^{p}(\mathbb{G})}-\|f\|_{L^{p}(\mathbb{G})}.$$
Finally, by using the inequality \eqref{LpSobolevinC} we have
$$\|\mathcal{R}(|x|f)\|_{L^{p}(\mathbb{G})}\geq\frac{Q-p}{p}\|f\|_{L^{p}(\mathbb{G})},$$
which implies the Hardy inequality \eqref{Hardy_p}.
\end{proof}
Now we establish weighted $L^{p}$-Sobolev type inequalities on the homogeneous group
$\mathbb{G}$.

\begin{thm}\label{L_p_weighted_th}
Let $\mathbb{G}$ be a homogeneous group
of homogeneous dimension $Q$ and let $\alpha\in \mathbb{R}$.
Then for all complex-valued functions $f\in C^{\infty}_{0}(\mathbb{G}\backslash\{0\}),$ $1<p<\infty,$ and any homogeneous quasi-norm $|\cdot|$ on $\mathbb{G}$ for $\alpha p \neq Q$ we have
\begin{equation}\label{L_p_weighted}
\left\|\frac{f}{|x|^{\alpha}}\right\|_{L^{p}(\mathbb{G})}\leq
\left|\frac{p}{Q-\alpha p}\right|\left\|\frac{1}{|x|^{\alpha}}\mathbb{E} f\right\|_{L^{p}(\mathbb{G})}.
\end{equation}
If $\alpha p\neq Q$ then the constant $\left|\frac{p}{Q-\alpha p}\right|$ is sharp.

For $\alpha p=Q$ we have
\begin{equation}\label{L_p_weighted_log}
\left\|\frac{f}{|x|^{\frac{Q}{p}}}\right\|_{L^{p}(\mathbb{G})}\leq
p\left\|\frac{\log|x|}{|x|^{\frac{Q}{p}}}\mathbb{E} f\right\|_{L^{p}(\mathbb{G})},
\end{equation}
where the constant $p$ is sharp.
\end{thm}
\begin{rem}\label{L_p_weighted_th_rem} In \cite[Theorem 1.2 and Theorem 7.1]{RSY17}, using these inequalities \eqref{L_p_weighted} and \eqref{L_p_weighted_log} we obtained an extended version of classical Caffarelli-Kohn-Nirenberg inequalities with respect to ranges of parameters, which are new in the Euclidean setting of $\mathbb{R}^{n}$, as well as on homogeneous groups. Moreover, our methods give an improvement by replacing the full gradient by the radial derivative. In \cite[Theorem 3.1]{LRY17}, applying Theorem \ref{L_p_weighted_th} the authors proved Hardy inequalities for the quadratic form of the Laplacian with the Landau Hamiltonian magnetic field.
\end{rem}
\begin{proof}[Proof of Theorem \ref{L_p_weighted_th}]
Using integration by parts, for $\alpha p \neq Q$ we obtain
\begin{multline*}
\int_{\mathbb{G}}\frac{|f(x)|^{p}}{|x|^{\alpha p}}dx=\int_{0}^{\infty}\int_{\wp}|f(ry)|^{p}r^{Q-1-\alpha p}d\sigma(y)dr\\
=-\frac{p}{Q-\alpha p}\int_{0}^{\infty} r^{Q-\alpha p} {\rm Re} \int_{\wp}|f(ry)|^{p-2} f(ry) \overline{\frac{df(ry)}{dr}}d\sigma(y)dr\\
\leq \left|\frac{p}{Q-\alpha p}\right|\int_{\mathbb{G}}\frac{|\mathbb{E}f(x)||f(x)|^{p-1}}{|x|^{\alpha p}}dx=
\left|\frac{p}{Q-\alpha p}\right|\int_{\mathbb{G}}\frac{|\mathbb{E}f(x)||f(x)|^{p-1}}{|x|^{\alpha+\alpha (p-1)}}dx.
\end{multline*}
By H\"{o}lder's inequality, it follows that
$$
\int_{\mathbb{G}}\frac{|f(x)|^{p}}{|x|^{\alpha p}}dx\leq \left|\frac{p}{Q-\alpha p}\right|\left(\int_{\mathbb{G}}\frac{|\mathbb{E}f(x)|^{p}}{|x|^{\alpha p}}dx\right)
^{\frac{1}{p}}\left(\int_{\mathbb{G}}\frac{|f(x)|^{p}}{|x|^{\alpha p}}dx\right)^{\frac{p-1}{p}},
$$
which gives \eqref{L_p_weighted}.

Now we show the sharpness of the constant. We need to check the equality
condition in above H\"older's inequality.
Let us consider the function
$$g(x)=\frac{1}{|x|^{C}}, $$
where $C\in\mathbb{R}, C\neq 0$ and $\alpha p\neq Q$. Then by a direct calculation we obtain
\begin{equation}\label{Holder_eq1}
\left|\frac{1}{C}\right|^{p}\left(\frac{|\mathbb{E}g(x)|}{|x|^{\alpha }}\right)^{p}=\left(\frac{|g(x)|^{p-1}}
{|x|^{\alpha (p-1)}}\right)^{\frac{p}{p-1}},
\end{equation}
which satisfies the equality condition in H\"older's inequality.
This gives the sharpness of the constant $\left|\frac{p}{Q-\alpha p}\right|$ in \eqref{L_p_weighted}.

Now let us prove \eqref{L_p_weighted_log}. Using integration by parts, we have
\begin{multline*}
\int_{\mathbb{G}}\frac{|f(x)|^{p}}{|x|^{Q}}dx=\int_{0}^{\infty}\int_{\wp}|f(ry)|^{p}r^{Q-1-Q}d\sigma(y)dr\\
=-p\int_{0}^{\infty} \log r {\rm Re} \int_{\wp}|f(ry)|^{p-2} f(ry) \overline{\frac{df(ry)}{dr}}d\sigma(y)dr\\
\leq p \int_{\mathbb{G}}\frac{|\mathbb{E}f(x)||f(x)|^{p-1}}{|x|^{Q}}|\log|x||dx=
p\int_{\mathbb{G}}\frac{|\mathbb{E}f(x)|\log|x|||}{|x|^{\frac{Q}{p}}}\frac{|f(x)|^{p-1}}{|x|^{\frac{Q(p-1)}{p}}}dx.
\end{multline*}
By H\"{o}lder's inequality, it follows that
$$
\int_{\mathbb{G}}\frac{|f(x)|^{p}}{|x|^{Q}}dx\leq p\left(\int_{\mathbb{G}}\frac{|\mathbb{E}f(x)|^{p}|\log|x||^{p}}{|x|^{Q}}dx\right)
^{\frac{1}{p}}\left(\int_{\mathbb{G}}\frac{|f(x)|^{p}}{|x|^{Q}}dx\right)^{\frac{p-1}{p}},
$$
which gives \eqref{L_p_weighted_log}.

Now we show the sharpness of the constant. We need to check the equality
condition in above H\"older's inequality.
Let us consider the function
$$h(x)=(\log|x|)^{C},$$
where $C\in\mathbb{R}$ and $C\neq 0$.
Then by a direct calculation we obtain
\begin{equation}\label{Holder_eq2}
\left|\frac{1}{C}\right|^{p}\left(\frac{|\mathbb{E}h(x)||\log|x||}{|x|^{\frac{Q}{p}}}\right)^{p}=\left(\frac{|h(x)|^{p-1}}
{|x|^{\frac{Q (p-1)}{p}}}\right)^{\frac{p}{p-1}},
\end{equation}
which satisfies the equality condition in H\"older's inequality.
This gives the sharpness of the constant $p$ in \eqref{L_p_weighted_log}.
\end{proof}

Let us consider separately the case $p=2$.

\begin{thm}\label{aSobolev}
Let $\mathbb{G}$ be a homogeneous group of homogeneous dimension $Q$ and
let $|\cdot|$ be any homogeneous quasi-norm on $\mathbb{G}$. Then for every complex-valued function $f\in C^{\infty}_{0}(\mathbb{G}\backslash\{0\})$
we have
\begin{multline}\label{awS}
\left\|\frac{1}{|x|^{\alpha}}\mathbb{E} f\right\|^{2}_{L^{2}(\mathbb{G})}=
\left(\frac{Q}{2}-\alpha\right)^{2}
\left\|\frac{f}{|x|^{\alpha}}\right\|^{2}_{L^{2}(\mathbb{G})}
+\left\|\frac{1}{|x|^{\alpha}}\mathbb{E} f+\frac{Q-2\alpha}{2|x|^{\alpha}}f
\right\|^{2}_{L^{2}(\mathbb{G})}
\end{multline}
for any $\alpha\in\mathbb{R}.$
\end{thm}
From \eqref{awS} one can get different equalities and inequalities. For example, for $\alpha=1$, we obtain the equality
\begin{equation}\label{47-0}
\left\|\frac{1}{|x|}\mathbb{E} f\right\|^{2}_{L^{2}(\mathbb{G})}=\left(\frac{Q-2}{2}\right)^{2}\left\|\frac{f}{|x|}
\right\|^{2}_{L^{2}(\mathbb{G})}+\left\|\frac{1}{|x|}\mathbb{E} f+\frac{Q-2}{2|x|}f
\right\|^{2}_{L^{2}(\mathbb{G})}.
\end{equation}

By dropping the nonnegative last term in \eqref{awS} we immediately get the following statement:

\begin{cor}\label{waSobolev}
Let $\mathbb{G}$ be a homogeneous group of homogeneous dimension $Q$ and let $|\cdot|$ be any homogeneous quasi-norm on $\mathbb{G}$. Let $\alpha\in\mathbb R$ and $Q-2\alpha\neq0$.
Then we have
\begin{equation}\label{awSoboleveq-g}
\left\|\frac{f}{|x|^{\alpha}}\right\|_{L^{2}(\mathbb{G})}\leq\frac{2}{|Q-2\alpha|}
\left\|\frac{1}{|x|^{\alpha}}\mathbb{E} f\right\|_{L^{2}(\mathbb{G})},
\end{equation}
for all complex-valued functions $f\in C^{\infty}_{0}(\mathbb{G}\backslash\{0\})$, where the constant in \eqref{awSoboleveq-g} is sharp and the equality is attained if and only if $f=0$.
\end{cor}

This statement on the constant and the equality follows by the same argument as that in Remark \ref{LpHardy}. We note a special case of $\alpha=1$, then \eqref{47-0}
gives the inequality
\begin{equation}\label{awSoboleveq}
\left\|\frac{f}{|x|^{}}
\right\|_{L^{2}(\mathbb{G})}\leq
\frac{2}{Q-2}\left\|\frac{1}{|x|}\mathbb{E} f\right\|_{L^{2}(\mathbb{G})}, \quad Q\geq 3,
\end{equation}
with sharp constant.

In the case $\alpha=0$, the identity \eqref{awS} recovers Part (iii) of Theorem \ref{RemSobolev}. However, in the proof of Theorem \ref{aSobolev} we will use Part (iii) of Theorem \ref{RemSobolev}.

\begin{proof}[Proof of Theorem \ref{aSobolev}]
For any $\alpha\in\mathbb{R}$ we note the following equality
\begin{equation}\label{EQ:eqE}
\frac{1}{|x|^{\alpha}}\mathbb{E} f=\mathbb{E} \frac{f}{|x|^{\alpha}}
+\alpha \frac{f}{|x|^{\alpha}}.
\end{equation}
Indeed, the equality \eqref{EQ:eqE} follows from
$$
\mathbb{E} \frac{f}{|x|^{\alpha}}=\frac{1}{|x|^{\alpha}}\mathbb{E} f+f \mathbb{E}  \frac{1}{|x|^{\alpha}}
$$
and utilising \eqref{dfdr} and \eqref{EQ:def-Euler},
$$
\mathbb{E}  \frac{1}{|x|^{\alpha}}=r\frac{d}{dr}\frac{1}{r^{\alpha}}=-\alpha\frac{1}{r^{\alpha}}=
-\alpha\frac{1}{|x|^{\alpha}},\quad r=|x|.
$$
Then we obtain
\begin{multline}
\left\|\frac{1}{|x|^{\alpha}}\mathbb{E} f\right\|^{2}_{L^{2}(\mathbb{G})}=\left\|\mathbb{E} \frac{f}{|x|^{\alpha}}
+\frac{\alpha f}{|x|^{\alpha}}
\right\|^{2}_{L^{2}(\mathbb{G})}\\
=\left\|\mathbb{E} \frac{f}{|x|^{\alpha}}
\right\|^{2}_{L^{2}(\mathbb{G})}+2\alpha{\rm Re}\int_{\mathbb{G}}
\mathbb{E} \left(\frac{f}{|x|^{\alpha}}\right)
\overline{\frac{f}{|x|^{\alpha}}}dx+\left\|\frac{\alpha f}{|x|^{\alpha}}
\right\|^{2}_{L^{2}(\mathbb{G})}.
\end{multline}
In \eqref{EQ:expL2} replacing $f$ by $\frac{f}{|x|^{\alpha}}$ and
using \eqref{EQ:eqE} we have that
\begin{equation}
\left\|\mathbb{E} \frac{f}{|x|^{\alpha}}
\right\|^{2}_{L^{2}(\mathbb{G})}=\left(\frac{Q}{2}\right)^{2}\left\|\frac{f}{|x|^{\alpha}}
\right\|^{2}_{L^{2}(\mathbb{G})}+\left\|\frac{1}{|x|^{\alpha}}\mathbb{E} f+\frac{Q-2\alpha}{2|x|^{\alpha}}f
\right\|^{2}_{L^{2}(\mathbb{G})}.
\end{equation}
Using formula \eqref{EQ:polar} for polar coordinates, one obtains
\begin{multline}
2\alpha{\rm Re}\int_{\mathbb{G}}
\mathbb{E} \left(\frac{f}{|x|^{\alpha}}\right)
\frac{\overline{f}}{|x|^{\alpha}}dx=2 \alpha{\rm Re}\int_{0}^{\infty}r^{Q-1}\int_{\wp}
r\frac{d}{dr}\left(\frac{f(ry)}
{r^{\alpha}}\right)\frac{\overline{f(ry)}}{r^{\alpha}}d\sigma(y)dr
\\
=\alpha\int_{0}^{\infty}r^{Q}\int_{\wp}
\frac{d}{dr}\left(\frac{|f(ry)|^{2}}
{r^{2\alpha}}\right)d\sigma(y)dr
=-\alpha Q\left\|\frac{f}{|x|^{\alpha}}
\right\|^{2}_{L^{2}(\mathbb{G})}.
\end{multline}
Summing up all above we arrive at
\begin{equation*}
\left\|\frac{1}{|x|^{\alpha}}\mathbb{E} f\right\|^{2}_{L^{2}(\mathbb{G})}=\left(\frac{Q}{2}-\alpha\right)^{2}\left\|\frac{f}{|x|^{\alpha}}
\right\|^{2}_{L^{2}(\mathbb{G})}+\left\|\frac{1}{|x|^{\alpha}}\mathbb{E} f+\frac{Q-2\alpha}{2|x|^{\alpha}}f
\right\|^{2}_{L^{2}(\mathbb{G})},
\end{equation*}
which implies \eqref{awS}.
\end{proof}

\subsection{Higher order Sobolev-Rellich inequalities}
\label{SEC:ho}

By iterating the established weighted Sobolev inequality \eqref{L_p_weighted} one obtains inequalities of higher order. Let us state the following:

\begin{cor} Let $\mathbb{G}$ be a homogeneous group of homogeneous dimension $Q$ and
let $|\cdot|$ be any homogeneous quasi-norm on $\mathbb{G}$. Let $1<p<\infty$, $k\in\mathbb{N}$ and $\alpha\in\mathbb{R}$ be such that $Q\neq \alpha p$. Then for any complex-valued function $f\in C^{\infty}_{0}(\mathbb{G}\backslash\{0\})$ we have
\begin{equation} \label{Lp_Sob_highorder}
\left\|\frac{f}{|x|^{\alpha}}\right\|_{L^{p}(\mathbb{G})}\leq
\left|\frac{p}{Q-\alpha p}\right|^{k}\left\|\frac{\mathbb{E}^{k}f}{|x|^{\alpha}}\right\|_{L^{p}(\mathbb{G})}.
\end{equation}
\end{cor}
\begin{rem} In the abelian case of the Euclidean space $\mathbb{G}=\mathbb{R}^{n}$, we have $Q=n$ and taking $|\cdot|$ to be the Euclidean norm, in the special case $\alpha=0$ and $k=1$, the unweighted Sobolev type inequality \eqref{Lp_Sob_highorder} was obtained in \cite[Theorem 1.1]{OS09}. In the case $k=2$ the inequality \eqref{Lp_Sob_highorder} can be thought of as a (weighted) Sobolev-Rellich type inequality.
\end{rem}

In the case $p=2$ an interesting feature is that we also have the exact formula for the remainder which provides the sharpness of the constants as well.

\begin{thm}\label{S-high} Let $\mathbb{G}$ be a homogeneous group of homogeneous dimension $Q$ and
let $|\cdot|$ be any homogeneous quasi-norm on $\mathbb{G}$. Let $\alpha\in\mathbb{R}$ and $k\in\mathbb N$ be such that $Q\neq2\alpha$.
Then for all complex-valued functions $f\in C^{\infty}_{0}(\mathbb{G}\backslash\{0\})$
the following inequality holds:
\begin{equation}\label{EQ:high-order1}
\left\|\frac{f}{|x|^{\alpha}}
\right\|_{L^{2}(\mathbb{G})}\leq
\left(\frac{2}{|Q-2\alpha|}\right)^{k}
\left\|\frac{1}{|x|^{\alpha}}\mathbb{E}^{k}f\right\|_{L^{2}(\mathbb{G})}.
\end{equation}
The constant in \eqref{EQ:high-order1} is sharp and the equality is attained if and only if $f=0$.

Furthermore, for all $k\in\mathbb N$ and $\alpha\in\mathbb{R}$, we have
\begin{multline}\label{equality-high-rem}
\left\|\frac{1}{|x|^{\alpha}}\mathbb{E}^{k}f\right\|^{2}_{L^{2}(\mathbb{G})}=
\left(\frac{Q-2\alpha}{2}\right)^{2k}\left\|\frac{f}{|x|^{\alpha}}
\right\|^{2}_{L^{2}(\mathbb{G})}\\
+\sum_{m=1}^{k}\left(\frac{Q-2\alpha}{2}\right)^{2k-2m}
\left\|\frac{1}{|x|^{\alpha}}\mathbb{E}^{m}f+ \frac{Q-2\alpha}{2|x|^{\alpha}}\mathbb{E}^{m-1}f\right\|^{2}_{L^{2}(\mathbb{G})}.
\end{multline}
\end{thm}

Although we often do not get sharp constants by iterative methods, since we have the formula \eqref{equality-high-rem}, we can apply it to prove that the iterative constant is sharp.
This may be a general feature of iterating Sobolev-Rellich type inequalities as the same phenomena was also investigated in $\mathbb{R}^{n}$ by Davies and Hinz \cite{Davies-Hinz}, but they have used very different methods for their analysis (see also \cite{Brez1} and \cite{Brez2}).

\begin{proof}[Proof of Theorem \ref{S-high}]
Let us iterate \eqref{awS}. For any $\alpha\in\mathbb{R}$ we start with
\begin{equation}\label{awS0}
\left\|\frac{1}{|x|^{\alpha}}\mathbb{E} f\right\|^{2}_{L^{2}(\mathbb{G})}=
\left(\frac{Q}{2}-\alpha\right)^{2}
\left\|\frac{f}{|x|^{\alpha}}\right\|^{2}_{L^{2}(\mathbb{G})}
+\left\|\frac{1}{|x|^{\alpha}}\mathbb{E} f+\frac{Q-2\alpha}{2|x|^{\alpha}}f
\right\|^{2}_{L^{2}(\mathbb{G})}.
\end{equation}
Putting $\mathbb{E} f$ instead of $f$ in \eqref{awS0} we obtain
\begin{equation}\label{awS0E}
\left\|\frac{1}{|x|^{\alpha}}\mathbb{E}^{2} f\right\|^{2}_{L^{2}(\mathbb{G})}=
\left(\frac{Q}{2}-\alpha\right)^{2}
\left\|\frac{\mathbb{E}f}{|x|^{\alpha}}\right\|^{2}_{L^{2}(\mathbb{G})}
+\left\|\frac{1}{|x|^{\alpha}}\mathbb{E}^{2} f+\frac{Q-2\alpha}{2|x|^{\alpha}}\mathbb{E}f
\right\|^{2}_{L^{2}(\mathbb{G})}.
\end{equation}
Combining \eqref{awS0E} with \eqref{awS0} we get
\begin{multline*}
\left\|\frac{1}{|x|^{\alpha}}\mathbb{E}^{2} f\right\|^{2}_{L^{2}(\mathbb{G})}=\left(\frac{Q}{2}-\alpha\right)^{4}
\left\|\frac{f}{|x|^{\alpha}}\right\|^{2}_{L^{2}(\mathbb{G})}\\+\left(\frac{Q}{2}-\alpha\right)^{2}
\left\|\frac{1}{|x|^{\alpha}}\mathbb{E} f+\frac{Q-2\alpha}{2|x|^{\alpha}}f
\right\|^{2}_{L^{2}(\mathbb{G})}\\
+\left\|\frac{1}{|x|^{\alpha}}\mathbb{E}^{2} f+\frac{Q-2\alpha}{2|x|^{\alpha}}\mathbb{E}f
\right\|^{2}_{L^{2}(\mathbb{G})}.
\end{multline*}
This iteration process implies
$$\left\|\frac{1}{|x|^{\alpha}}\mathbb{E}^{k}f\right\|^{2}_{L^{2}(\mathbb{G})}=
\left(\frac{Q-2\alpha}{2}\right)^{2k}\left\|\frac{f}{|x|^{\alpha}}
\right\|^{2}_{L^{2}(\mathbb{G})}$$
$$+\sum_{m=1}^{k}\left(\frac{Q-2\alpha}{2}\right)^{2k-2m}
\left\|\frac{1}{|x|^{\alpha}}\mathbb{E}^{m}f+ \frac{Q-2\alpha}{2|x|^{\alpha}}\mathbb{E}^{m-1}f\right\|^{2}_{L^{2}(\mathbb{G})}, \quad k=1,2,\ldots.$$
By dropping nonnegative terms, we obtain
\begin{equation}\label{awSoboleveq-high}
\left\|\frac{1}{|x|^{\alpha}}\mathbb{E}^{k}f\right\|^{2}_{L^{2}(\mathbb{G})}\geq
\left(\frac{Q-2\alpha}{2}\right)^{2k}\left\|\frac{f}{|x|^{\alpha}}
\right\|^{2}_{L^{2}(\mathbb{G})}.
\end{equation}
If $Q\neq2\alpha$, this means
\begin{equation}\label{awSoboleveq-high}
\left\|\frac{f}{|x|^{\alpha}}
\right\|^{2}_{L^{2}(\mathbb{G})} \leq\left(\frac{2}{Q-2\alpha}\right)^{2k}\left\|\frac{1}{|x|^{\alpha}}\mathbb{E}^{k}f\right\|^{2}_{L^{2}(\mathbb{G})}
,
\end{equation}
which implies \eqref{EQ:high-order1}.
Now let us show the sharpness of the constant in \eqref{EQ:high-order1}.
The equality $$\frac{1}{|x|^{\alpha}}\mathbb{E}^{m}f+ \frac{Q-2\alpha}{2|x|^{\alpha}}\mathbb{E}^{m-1}f=0$$
can be restated as $$\mathbb{E} (\mathbb{E}^{m-1}f)+
\frac{Q-2\alpha}{2}\mathbb{E}^{m-1}f=0,$$ and
by Lemma \ref{L:Euler} it follows that
$\mathbb{E}^{m-1}f$ is positively homogeneous of degree $-\frac{Q}{2}+\alpha$.
Thus, if $f$ is positively homogeneous of degree
$m-1-\frac{Q}{2}+\alpha$, then all the remainder terms vanish. Since this can be approximated by functions in
$C^{\infty}_{0}(\mathbb{G}\backslash\{0\})$, the constant $\left(\frac{2}{Q-2\alpha}\right)^{2k}$ is sharp.
Even if it were attained, then it would be on functions $f$ which are positively homogeneous of
degree $m-1-\frac{Q}{2}+\alpha$. In this case $\frac{f}{|x|^{\alpha+m-1}}$ would be positively homogeneous of
degree $-\frac{Q}{2}$, and these are in $L^{2}$ if and only if they are zero.
\end{proof}

\section{Euler-Hilbert-Sobolev space on homogeneous groups}
\label{SEC:EHS}
In this section we introduce an Euler-Hilbert-Sobolev space on homogeneous groups.
First let us define the Euler-Sobolev function space by
\begin{equation}\label{ESdef}\mathfrak{L}^{k,p}(\mathbb{G})\equiv \overline{C^{\infty}_{0}(\mathbb{G}\backslash\{0\})}^{\|\cdot\|_{\mathfrak{L}^{k,p}(\mathbb{G})}},\;k\in\mathbb{Z},
\end{equation}
where
$$\|f\|_{\mathfrak{L}^{k,p}(\mathbb{G})}:=\|\mathbb{E}^{k} f\|_{L^{p}(\mathbb{G})}.$$
By \eqref{ESdef}, it is easy to see that the higher order Sobolev-Rellich inequality \eqref{Lp_Sob_highorder} with $\alpha=0$ holds for all functions $f\in \mathfrak{L}^{k,p}(\mathbb{G})$:
\begin{equation}\label{ES1}
\|f\|_{L^{p}(\mathbb{G})}\leq\left(\frac{p}{Q}\right)^{k}\|\mathbb{E}^{k}f\|_{L^{p}(\mathbb{G})}, \;1<p<\infty, \;k\in\mathbb{N}.
\end{equation}
By taking into account the definition of the Euler-Sobolev function space \eqref{ESdef} and higher order Sobolev-Rellich inequality \eqref{ES1}, we obtain the following Proposition \ref{Banach1}:
\begin{prop} \label{Banach1} Let $\mathbb{G}$ be a homogeneous group of homogeneous dimension $Q$ and let $1<p<\infty$. Then the semi-normed space $(\mathfrak{L}^{k,p}(\mathbb{G}),\|\cdot\|_{\mathfrak{L}^{k,p}(\mathbb{G})})$, $k\in\mathbb{Z}$ is a complete space. The norm of the embedding operator $\iota : (\mathfrak{L}^{k,p}(\mathbb{G}), \|\cdot\|_{\mathfrak{L}^{k,p}(\mathbb{G})})\hookrightarrow (L^{p}(\mathbb{G}), \|\cdot\|_{L^{p}(\mathbb{G})})$ satisfies
\begin{equation} \label{emb}
\|\iota\|_{\mathfrak{L}^{k,p}(\mathbb{G})\rightarrow L^{p}(\mathbb{G})}\leq\left(\frac{p}{Q}\right)^k,\;k\in \mathbb{N},
\end{equation}
where we understand the embedding $\iota$ as an embedding of semi-normed subspace of $L^{p}(\mathbb{G})$.
\end{prop}
By using Lemma \ref{L2:Komatsu} we can define fractional powers of the operator $\mathbb{A}=\mathbb{E}\mathbb{E}^{*}$ as in \cite[Chapter 5]{Martinez-Sanz:fractional} and we denote
$$|\mathbb{E}|^{\beta}:=\mathbb{A}^{\frac{\beta}{2}}, \;\;\;\;\;\beta\in \mathbb{C}.$$
For a brief account of the relevant theory of fractional powers we refer to \cite[App.A]{FR}.
\begin{thm} \label{fracineq}
Let $\mathbb{G}$ be a homogeneous group
of homogeneous dimension $Q$, $\beta\in \mathbb{C_{+}}$ and let $k>\frac{{\rm Re} \beta}{2}$ be a positive integer. Then for all complex-valued functions $f\in C_{0}^{\infty}(\mathbb{G}\backslash\{0\})$ we have
\begin{equation}\label{fracineq1}
\left\|f\right\|_{L^{2}(\mathbb{G})}\leq C(k-\frac{\beta}{2},k)\left(\frac{2}{Q}\right)^{{\rm Re}\beta}\left\|\mathbb{|E|}^{\beta} f\right\|_{L^{2}(\mathbb{G})},
\end{equation}
where \begin{equation}\label{fracineq2}C(\beta,k)=\frac{\Gamma(k+1)}{|\Gamma(\beta)\Gamma(k-\beta)|}\frac{2^{k-{\rm Re}\beta}}{{\rm Re}\beta (k-{\rm Re}\beta)}.
\end{equation}
\end{thm}
\begin{proof}[Proof of Theorem \ref{fracineq}]
By using \cite[Proposition 7.2.1, p.176]{Martinez-Sanz:fractional} we obtain
\begin{equation}\label{fracineq01}
\left\|\mathbb{|E|}^{-\beta} f\right\|_{L^{2}(\mathbb{G})}\leq C(k-\frac{\beta}{2},k)\left\|f\right\|^{1-\frac{{\rm Re} \beta}{2k}}_{L^{2}(\mathbb{G})}\left\|\mathbb{A}^{-k}f\right\|^{\frac{{\rm Re} \beta}{2k}}_{L^{2}(\mathbb{G})}.
\end{equation}
By Corollary \ref{AE} and \eqref{EQ:high-order1} with $\alpha=0$ it follows that
\begin{multline*}C(k-\frac{\beta}{2},k)\left\|f\right\|^{1-\frac{{\rm Re} \beta}{2k}}_{L^{2}(\mathbb{G})}\left\|\mathbb{A}^{-k}f\right\|^{\frac{{\rm Re} \beta}{2k}}_{L^{2}(\mathbb{G})}\\ \leq C(k-\frac{\beta}{2},k)\left\|f\right\|^{1-\frac{{\rm Re} \beta}{2k}}_{L^{2}(\mathbb{G})}\left(\frac{4}{Q^{2}}\right)^{\frac{{\rm Re} \beta }{2}}\left\|f\right\|^{\frac{{\rm Re} \beta}{2k}}_{L^{2}(\mathbb{G})}\\
=C(k-\frac{\beta}{2},k)\left(\frac{4}{Q^{2}}\right)^{\frac{{\rm Re} \beta }{2}}\left\|f\right\|_{L^{2}(\mathbb{G})},
\end{multline*}
which combined with \eqref{fracineq01} implies \eqref{fracineq1}.
\end{proof}
Now we define the Euler-Hilbert-Sobolev function space by
\begin{equation}\label{EHSdef}
\mathbb{H}^{\beta}(\mathbb{G})\equiv\overline{C^{\infty}_{0}(\mathbb{G}\backslash\{0\})}^{\|\cdot\|_{\mathbb{H}^{\beta}(\mathbb{G})}},
\end{equation}
where
$$\|f\|_{\mathbb{H}^{\beta}(\mathbb{G})}:=\|\mathbb{|E|}^{\beta} f\|_{L^{2}(\mathbb{G})}.$$
By \eqref{EHSdef} we obtain the inequality \eqref{fracineq1} for all $f\in \mathbb{H}^{\beta}(\mathbb{G})$:
\begin{equation}\label{EHS1}
\left\|f\right\|_{L^{2}(\mathbb{G})}\leq C(k-\frac{\beta}{2},k)\left(\frac{2}{Q}\right)^{{\rm Re}\beta}\left\|\mathbb{|E|}^{\beta} f\right\|_{L^{2}(\mathbb{G})},
\end{equation}
where $\beta\in \mathbb{C_{+}},\; k>\frac{{\rm Re} \beta}{2}, \;k\in \mathbb{N}$ and $C(k-\frac{\beta}{2},k)$ is given by \eqref{fracineq2}.

By taking into account the definition of the Euler-Hilbert-Sobolev function space \eqref{EHSdef} and inequality \eqref{fracineq1}, we obtain the following Proposition \ref{Banach2}:
\begin{prop} \label{Banach2} The semi-normed space $(\mathbb{H}^{\beta},\|\cdot\|_{\mathbb{H}^{\beta}})$, $\beta\in\mathbb{C}$ is a complete space. Moreover, the norm of the embedding operator $\iota : (\mathbb{H}^{\beta}, \|\cdot\|_{\mathbb{H}^{\beta}})\hookrightarrow (L^{2}, \|\cdot\|_{L^{2}})$ satisfies
\begin{equation} \label{emb1}
\|\iota\|_{\mathbb{H}^{\beta}(\mathbb{G})\rightarrow L^{2}(\mathbb{G})}\leq C\left(k-\frac{\beta}{2},k\right)\left(\frac{2}{Q}\right)^{{\rm Re} \beta }, \quad \beta\in \mathbb{C_{+}}, \;k>\frac{{\rm Re} \beta}{2}, \;k\in \mathbb{N},
\end{equation}
where we understand the embedding $\iota$ as an embedding of semi-normed subspace of $L^{2}(\mathbb{G})$.
\end{prop}

\section{Poincar\'{e} type inequality on $\mathbb{G}$}
\label{SEC:Poincare}
In this section we establish Poincar\'{e} type inequality on the homogeneous group
$\mathbb{G}$. Before stating our results, we introduce some notation. Let $\Omega \subset \mathbb{G}$ be an open set and let $\widehat{\mathfrak{L}}_{0}^{1,p}(\Omega)$ be the completion of $C^{\infty}_{0}(\Omega\backslash\{0\})$ with respect to
$$\|f\|_{\widehat{\mathfrak{L}}^{1,p}(\Omega)}=\|f\|_{L^{p}(\Omega)}+\|\mathbb{E}f\|_{L^{p}(\Omega)}, \;\;1<p<\infty.$$

\begin{thm}\label{Po}
Let $\Omega$ be a bounded open subset of $\;\mathbb{G}$. If $\;1<p<\infty, \;f\in \widehat{\mathfrak{L}}_{0}^{1,p}(\Omega)$ and $\mathcal{R}f\equiv\frac{1}{|x|}\mathbb{E}f\in L^{p}(\Omega)$, then we have the following Poincar\'{e} type inequality on $\Omega \subset \mathbb{G}$:
\begin{equation}\label{Po01}
\|f\|_{L^{p}(\Omega)}\leq  \frac{R p}{Q}\|\mathcal{R}f\|_{L^{p}(\Omega)}=\frac{R p}{Q}\left\|\frac{1}{|x|}\mathbb{E}f\right\|_{L^{p}(\Omega)},
\end{equation}
where $R=\underset{x\in \Omega}{\rm sup}|x|$.
\end{thm}
In order to prove Theorem \ref{Po}, we first show the following proposition.
\begin{prop}\label{Poprop}
Let $\Omega \subset \mathbb{G}$ be an open set. If $1<p<\infty, \;f\in \widehat{\mathfrak{L}}_{0}^{1,p}(\Omega)$ and $\mathbb{E}f\in L^{p}(\Omega)$, then we have
\begin{equation} \label{Po1}
\|f\|_{L^{p}(\Omega)}\leq \frac{p}{Q}\|\mathbb{E}f\|_{L^{p}(\Omega)}.
\end{equation}
\end{prop}
\begin{proof}[Proof of Proposition \ref{Poprop}]
Let $\zeta:\mathbb{R}\rightarrow\mathbb{R}$ be an even, smooth function satisfying
\begin{itemize}
\item
$0\leq\zeta\leq1,$
\item
$\zeta(r)=1 \;\;{\rm if}\;\; |r|\leq1,$
\item
$\zeta(r)=0 \;\;{\rm if}\;\; |r|\geq2.$
\end{itemize}
For $\lambda>0$, we set $\zeta_{\lambda}(x):=\zeta(\lambda|x|)$.
We have the inequality \eqref{Po1} for $f\in C_{0}^{\infty}(\mathbb{G}\backslash\{0\})$ by \eqref{LpSobolevinC}. There exists some $\{f_{\ell}\}_{\ell=1}^{\infty} \in C_{0}^{\infty}(\Omega\backslash\{0\})$ such that $f_{\ell}\rightarrow f$ in $\widehat{\mathfrak{L}}_{0}^{1,p}(\Omega)$ as $\ell\rightarrow\infty$. Let $\lambda>0$. From \eqref{LpSobolevinC} we obtain
$$\|\zeta_{\lambda}f_{\ell}\|_{L^{p}(\Omega)}\leq\frac{p}{Q}\left(\|(\mathbb{E}\zeta_{\lambda})f_{\ell}\|_{L^{p}(\Omega)}+
\|\zeta_{\lambda}(\mathbb{E}f_{\ell})\|_{L^{p}(\Omega)}\right)$$
for all $\ell\geq1$. It is easy to see that
$$\lim_{\ell\rightarrow\infty}\zeta_{\lambda}f_{\ell}=\zeta_{\lambda}f,$$
$$\lim_{\ell\rightarrow\infty}(\mathbb{E}\zeta_{\lambda})f_{\ell}=(\mathbb{E}\zeta_{\lambda})f,$$
$$\lim_{\ell\rightarrow\infty}\zeta_{\lambda}(\mathbb{E}f_{\ell})=\zeta_{\lambda}(\mathbb{E} f)$$
in $L^{p}(\Omega)$. These properties imply that
$$\|\zeta_{\lambda}f\|_{L^{p}(\Omega)}\leq\frac{p}{Q}\left\{\|(\mathbb{E}\zeta_{\lambda})f\|_{L^{p}(\Omega)}+
\|\zeta_{\lambda}(\mathbb{E}f)\|_{L^{p}(\Omega)}\right\}.$$
Since
$$|(\mathbb{E}\zeta_{\lambda})(x)|\leq \begin{cases}
\sup|\mathbb{E}\zeta|, \;\;{\rm if}\;\; \lambda^{-1}<|x|<2\lambda^{-1};\\
0, \;\; {\rm otherwise},
\end{cases}$$
we obtain \eqref{Po1} in the limit as $\lambda\rightarrow 0$.
\end{proof}
\begin{proof}[Proof of Theorem \ref{Po}] Since $R=\underset{x\in \Omega}{\rm sup}|x|$ using Proposition \ref{Poprop} we obtain
\begin{equation*}\|f\|_{L^{p}(\Omega)}\leq \frac{p}{Q}\|\mathbb{E}f\|_{L^{p}(\Omega)}\\ \leq
\frac{R p}{Q}\|\mathcal{R}f\|_{L^{p}(\Omega)}=\frac{R p}{Q}\left\|\frac{1}{|x|}\mathbb{E}f\right\|_{L^{p}(\Omega)},
\end{equation*}
which gives \eqref{Po01}.
\end{proof}

\section{Sobolev-Lorentz-Zygmund spaces}
\label{SLZ}
In this section, we consider applications of critical Hardy type inequalities to function spaces. The function spaces below extend some known results in the abelian case $\mathbb{R}^{n}$, see e.g. \cite{MOW15c}.

We define the Lorentz type spaces by
$$L_{|\cdot|,Q,p,q}(\mathbb{G}):=\{f\in L^{1}_{loc}(\mathbb{G}):\|f\|_{L_{|\cdot|,Q,p,q}(\mathbb{G})}<\infty\}, \quad 0\leq p,q\leq\infty,$$
where
$$\|f\|_{L_{|\cdot|,Q,p,q}(\mathbb{G})}:=\left(\int_{\mathbb{G}}(|x|^{\frac{Q}{p}}|f(x)|)^{q}
\frac{1}{|x|^{Q}}dx\right)^{\frac{1}{q}}.$$
We assume that the homogeneous dimension $Q$ and the homogeneous quasi-norm $|\cdot|$ are fixed. Therefore, we can use the short notation
$$L_{p,q}(\mathbb{G}):=L_{|\cdot|,Q,p,q}(\mathbb{G}).$$
The Lorentz-Zygmund spaces on $\mathbb{G}$ can be defined by
$$L_{p,q,\lambda}(\mathbb{G}):=\{f\in L^{1}_{loc}(\mathbb{G}):\|f\|_{L_{p,q,\lambda}(\mathbb{G})}<\infty\}, \quad 0\leq p,q\leq\infty, \; \lambda \in \mathbb{R},$$
where
$$\|f\|_{L_{p,q,\lambda}(\mathbb{G})}:=\underset{R>0}{\rm sup}\left(\int_{\mathbb{G}}\left(|x|^{\frac{Q}{p}}
\left|\log\frac{R}{|x|}\right|^{\lambda}|f(x)|\right)^{q}
\frac{1}{|x|^{Q}}dx\right)^{\frac{1}{q}}.$$
Then we define the Sobolev-Lorentz-Zygmund spaces by
$$
W^{1}L_{p,q,\lambda}(\mathbb{G}):=\left\{f\in L_{p,q,\lambda}(\mathbb{G}):\frac{1}{|x|}\mathbb{E} f \in L_{p,q,\lambda}(\mathbb{G})\right\},
$$
endowed with the norm
$$\|\cdot\|_{W^{1}L_{p,q,\lambda}(\mathbb{G})}:=\|\cdot\|_{L_{p,q,\lambda}(\mathbb{G})}+
\left\|\frac{1}{|x|}\mathbb{E}\cdot\right\|_{L_{p,q,\lambda}(\mathbb{G})},$$
and $W^{1}_{0}L_{p,q,\lambda}(\mathbb{G}):=\overline{C_{0}^{\infty}(\mathbb{G})}^{\|\cdot\|_{W^{1}L_{p,q,\lambda}(\mathbb{G})}}$. The Lorentz-Zygmund spaces involving the double logarithmic weights are introduced by
$$L_{p,q,\lambda_{1},\lambda_{2}}(\mathbb{G}):=\{f\in L^{1}_{loc}(\mathbb{G}): \|f\|_{L_{p,q,\lambda_{1},\lambda_{2}}(\mathbb{G})}<\infty\},$$
where $\lambda_{1},\lambda_{2} \in \mathbb{R}$ and
$$
\|f\|_{L_{p,q,\lambda_{1},\lambda_{2}}(\mathbb{G})}:=
$$
$$\underset{R>0}{\rm sup}\left(\int_{\mathbb{G}}\left(|x|^{\frac{Q}{p}}
\left|\log\frac{R}{|x|}\right|^{\lambda_{1}}\left|\log\left|\log\frac{R}{|x|}\right|\right|^{\lambda_{2}}
|f(x)|\right)^{q}
\frac{dx}{|x|^{Q}}\right)^{\frac{1}{q}}.
$$
\begin{rem} The space $L_{p,q,\lambda_{1},\lambda_{2}}(\mathbb{G})$ extends the spaces $L_{p,q,\lambda}(\mathbb{G})$ and $L_{p,q}(\mathbb{G})$ in the sence that $L_{p,q,\lambda,0}(\mathbb{G})=L_{p,q,\lambda}(\mathbb{G})$ and $L_{p,q,0,0}(\mathbb{G})=L_{p,q}(\mathbb{G})$.
\end{rem}
Similarly, the Sobolev-Lorentz-Zygmund spaces $W^{1}L_{p,q,\lambda_{1},\lambda_{2}}(\mathbb{G})$ are defined by
\begin{equation}\label{SLZdef}
W^{1}L_{p,q,\lambda_{1},\lambda_{2}}(\mathbb{G}):=\left\{f\in L_{p,q,\lambda_{1},\lambda_{2}}(\mathbb{G}):\frac{1}{|x|}\mathbb{E}f \in L_{p,q,\lambda_{1},\lambda_{2}}(\mathbb{G})\right\},
\end{equation}
endowed with the norm
$$\|\cdot\|_{W^{1}L_{p,q,\lambda_{1},\lambda_{2}}(\mathbb{G})}:=\|\cdot\|_{L_{p,q,\lambda_{1},\lambda_{2}}(\mathbb{G})}+
\left\|\frac{1}{|x|}\mathbb{E}\cdot\right\|_{L_{p,q,\lambda_{1},\lambda_{2}}(\mathbb{G})},$$
and
\begin{equation}\label{SLZdef1}W^{1}_{0}L_{p,q,\lambda_{1},\lambda_{2}}(\mathbb{G}):=\overline{C_{0}^{\infty}
(\mathbb{G})}^{\|\cdot\|_{W^{1}L_{p,q,\lambda_{1},\lambda_{2}}(\mathbb{G})}}.\end{equation}

Now we introduce the Lorentz-Zygmund type spaces $\mathfrak{L}_{p,q,\lambda}(\mathbb{G})$ taking into account the special behavior of functions,
$$\mathfrak{L}_{p,q,\lambda}(\mathbb{G}):=\{f \in L^{1}_{loc}(\mathbb{G}):\|f\|_{\mathfrak{L}_{p,q,\lambda}(\mathbb{G})}< \infty \}, \quad \lambda \in \mathbb{R},$$
where
$$\|f\|_{\mathfrak{L}_{p,q,\lambda}(\mathbb{G})}:=\underset{R>0}{\rm sup}\left(\int_{\mathbb{G}}\left(|x|^\frac{Q}{p}
\left|\log\frac{R}{|x|}\right|^{\lambda}|f-f_{R}|\right)^{q}\frac{dx}{|x|^{Q}}\right)^{\frac{1}{q}}.$$
For $p=\infty$ we define
$$\|f\|_{\mathfrak{L}_{\infty,q,\lambda}(\mathbb{G})}:=\underset{R>0}{\rm sup}\left(\int_{\mathbb{G}}\left(
\left|\log\frac{R}{|x|}\right|^{\lambda}|f-f_{R}|\right)^{q}\frac{dx}{|x|^{Q}}\right)^{\frac{1}{q}},$$
where $f_{R}(x):=f(R\frac{x}{|x|})$.

Moreover, we define the Lorentz-Zygmund type spaces $\mathfrak{L}_{p,q,\lambda_{1},\lambda_{2}}(\mathbb{G})$ by
\begin{equation}\label{LZdef}\mathfrak{L}_{p,q,\lambda_{1},\lambda_{2}}(\mathbb{G}):=\{f \in L^{1}_{loc}(\mathbb{G}):\|f\|_{\mathfrak{L}_{p,q,\lambda_{1},\lambda_{2}}(\mathbb{G})}< \infty \},\end{equation}
where
\begin{multline*}
\|f\|_{\mathfrak{L}_{p,q,\lambda_{1},\lambda_{2}}(\mathbb{G})}:=
\underset{R>0}{\rm sup}\left(\int_{\mathbb{G}}\left(|x|^\frac{Q}{p}
\left|\log\frac{e R}{|x|}\right|^{\lambda_{1}}\left|\log\left|\log\frac{e R}{|x|}\right|\right|^{\lambda_{2}} \right.\right.
\\ \left.\left. \times
\left(\chi_{B(0,eR)}(x)|f-f_{R}|
+\chi_{B^{c}(0,eR)}(x)|f-f_{e^{2}R}|\right)\right)^{q}\frac{dx}{|x|^{Q}}\right)^{\frac{1}{q}},
\end{multline*}
$$\chi_{B(0,eR)}(x)=\begin{cases}
1, \;\; x\in B(0,eR);\\
0, \;\; x\notin B(0,eR).
\end{cases}$$
For $p=\infty$ we define
\begin{multline*}
\|f\|_{\mathfrak{L}_{\infty,q,\lambda_{1},\lambda_{2}}(\mathbb{G})}:=
\underset{R>0}{\rm sup}\left(\int_{\mathbb{G}}\left(
\left|\log\frac{e R}{|x|}\right|^{\lambda_{1}}\left|\log\left|\log\frac{e R}{|x|}\right|\right|^{\lambda_{2}} \right.\right.
\\\left. \left. \times \left(\chi_{B(0,eR)}(x)|f-f_{R}|
+\chi_{B^{c}(0,eR)}(x)|f-f_{e^{2}R}|\right)\right)^{q}\frac{dx}{|x|^{Q}}\right)^{\frac{1}{q}}.
\end{multline*}
\begin{thm}\label{SLZ0} Let $\mathbb{G}$ be a homogeneous group of homogeneous dimension $Q$ and
let $|\cdot|$ be any homogeneous quasi-norm on $\mathbb{G}$. Let $1<\gamma<\infty$ and $\max\{1,\gamma-1\}<q<\infty$. Then the continuous embedding
\begin{equation*}W^{1}_{0}L_{Q,q,\frac{q-1}{q},
\frac{q-\gamma}{q}}(\mathbb{G})\hookrightarrow \mathfrak{L}_{\infty,q,-\frac{1}{q},
-\frac{\gamma}{q}}(\mathbb{G})
\end{equation*}
holds. In particular, for all $f\in W^{1}_{0}L_{{Q,q,\frac{q-1}{q},
\frac{q-\gamma}{q}}}(\mathbb{G})$ and for any $R>0$ the following inequality holds
\begin{multline}\label{SLZ1}
\left(\int_{\mathbb{G}}\frac{\chi_{B(0,eR)}(x)|f-f_{R}|^{q}+
\chi_{B^{c}(0,eR)}(x)|f-f_{e^{2}R}|^{q}}
{\left|\log\left|\log\frac{e R}{|x|}\right|\right|^{\gamma}\left|\log\frac{e R}{|x|}\right|}\frac{dx}{|x|^{Q}}\right)^{\frac{1}{q}}\\
\leq \frac{q}{\gamma-1}\left(\int_{\mathbb{G}}|x|^{q-Q}
\left|\log\frac{e R}{|x|}\right|^{q-1}\left|\log\left|\log\frac{e R}{|x|}\right|\right|^{q-\gamma}\left|\frac{1}{|x|}\mathbb{E}f\right|
^{q}dx\right)^{\frac{1}{q}},
\end{multline}
where the embedding constant $\frac{q}{\gamma-1}$ is sharp and $f_{R}(x):=f(R\frac{x}{|x|})$.
\end{thm}
\begin{rem}
Despite the integrand on the right-hand side of \eqref{SLZ1} has singularities for $|x|=R, |x|=eR$, and $|x|=e^{2}R$ we do not need to subtract the boundary value of functions on $|x|=e R$ on the left-hand side.
\end{rem}
\begin{rem} In the abelian case of the Euclidean space $\mathbb{G}=\mathbb{R}^{n}$, we have $Q=n$ and taking $|\cdot|$ to be the Euclidean norm, Theorem \ref{SLZ0} implies \cite[Theorem 1.2]{MOW15c}.
\end{rem}
In order to prove Theorem \ref{SLZ0}, let us first present the following proposition.
\begin{prop} \label{SLZ01} Let $\mathbb{G}$ be a homogeneous group of homogeneous dimension $Q$ and
let $|\cdot|$ be any homogeneous quasi-norm on $\mathbb{G}$. Let $1<\gamma<\infty$ and $\max\{1,\gamma-1\}<q<\infty$. Then for all $f\in C_{0}^{\infty}(\mathbb{G})$ and any $R>0$ the following inequality holds
\begin{multline}\label{SLZ2}
\left(\int_{B(0,e R)}\frac{|f-f_{R}|^{q}}{\left|\log\left|\log\frac{e R}{|x|}\right|\right|^{\gamma}\left|\log\frac{e R}{|x|}\right|}\frac{dx}{|x|^{Q}}\right)^{\frac{1}{q}}\\
\leq\frac{q}{\gamma-1}\left(\int_{B(0,e R)}|x|^{q-Q}
\left|\log\frac{e R}{|x|}\right|^{q-1}\left|\log\left|\log\frac{e R}{|x|}\right|\right|^{q-\gamma}\left|\frac{1}{|x|}\mathbb{E}f\right|
^{q}dx\right)^{\frac{1}{q}}.
\end{multline}
\end{prop}
\begin{proof}[Proof of Proposition \ref{SLZ01}]
First of all we consider the integrals in \eqref{SLZ2} restricted to $B(0,R)$. Using polar coordinates and integration by parts, we obtain
\begin{multline*}
\int_{B(0,R)}\frac{|f-f_{R}|^{q}}{\left|\log\left|\log\frac{e R}{|x|}\right|\right|^{\gamma}\left|\log\frac{e R}{|x|}\right|}\frac{dx}{|x|^{Q}}\\
=\int_{0}^{R}\frac{1}{r(\log\frac{e R}{r})(\log(\log \frac{e R}{r}))^{\gamma}}\int_{\wp}|f(r y)-f(R y)|^{q}d\sigma(y)dr\\
=\frac{1}{\gamma-1}\left[\left(\log\left(\log\frac{e R}{r}\right)\right)^{-\gamma+1}\int_{\wp}|f(r y)-f(R y)|^{q}d\sigma(y)\right]^{r=R}_{r=0}\\
-\frac{1}{\gamma-1}\int_{0}^{R}\left(\log\left(\log\frac{e R}{r}\right)\right)^{-\gamma+1}
\frac{d}{dr}\int_{\wp}|f(r y)-f(R y)|^{q}d\sigma(y)dr\\
=-\frac{q}{\gamma-1}\int_{0}^{R}\left(\log\left(\log\frac{e R}{r}\right)\right)^{-\gamma+1}
{\rm Re}\int_{\wp}|f(r y)-f(R y)|^{q-2}\\
\times(f(r y)-f(R y))\overline{\frac{df(ry)}{dr}}d\sigma(y)dr,
\end{multline*}
where $\sigma$ is the Borel measure on $\wp, q-\gamma+1>0$, so that the boundary term at $r=R$ vanishes due to inequalities
\begin{multline*}\log\left(\log \frac{e R}{r}\right)=\int_{1}^{\log \frac{e R}{r}}\frac{dt}{t}\geq \frac{\log\frac{e R}{r}-1}{\log\frac{e R}{r}}=\frac{\log\frac{R}{r}}{\log\frac{e R}{r}}\\
=\frac{1}{\log\frac{e R}{r}}\int_{1}^{\frac{R}{r}}\frac{dt}{t}\geq\frac{1}{\log\frac{e R}{r}}\frac{\frac{R}{r}-1}{\frac{R}{r}}=\frac{R-r}{R \log \frac{e R}{r}}
\end{multline*}
and
$$|f(ry)-f(Ry)|\leq C (R-r)$$
for $0<r\leq R$.
It follows that
\begin{multline*}
\int_{0}^{R}\frac{1}{r(\log\frac{e R}{r})(\log(\log \frac{e R}{r}))^{\gamma}}\int_{\wp}|f(r y)-f(R y)|^{q}d\sigma(y)dr\\
\leq \frac{q}{\gamma-1}\int_{0}^{R}\frac{1}{(\log(\log\frac{e R}{r}))^{\gamma-1}}\int_{\wp}|f(r y)-f(R y)|^{q-1}\left|\frac{df(r y)}{dr}\right|d\sigma(y)dr\\
=\frac{q}{\gamma-1}\int_{0}^{R}\frac{1}
{r^{\frac{q-1}{q}}(\log\frac{e R}{r})^{\frac{q-1}{q}} r^{-\frac{q-1}{q}}(\log\frac{e R}{r})^{-\frac{q-1}{q}}}\\
\times \frac{1}{(\log(\log\frac{e R}{r}))^{\frac{(q-1)\gamma}{q}}
(\log(\log \frac{e R}{r}))^{\frac{\gamma-q}{q}}}\\
\times\int_{\wp}|f(r y)-f(R y)|^{q-1}\left|\frac{df(ry)}{dr}\right|d\sigma(y)dr.
\end{multline*}
By the H\"{o}lder inequality, we obtain
\begin{multline*}
\int_{0}^{R}\frac{1}{r(\log\frac{e R}{r})(\log(\log \frac{e R}{r}))^{\gamma}}\int_{\wp}|f(r y)-f(R y)|^{q}d\sigma(y)dr\\
\leq \frac{q}{\gamma-1} \left(\int_{0}^{R} \int_{\wp}
\frac{|f(r y)-f(R y)|^{q}}{r(\log \frac{e R}{r})(\log(\log \frac{e R}{r}))^{\gamma}}d\sigma(y)dr\right)^\frac{q-1}{q}\\
\times \left(\int_{0}^{R} \int_{\wp}r^{q-1}\left(\log \frac{e R}{r}\right)^{q-1}\left(\log\left(\log\frac{e R}{r}\right)\right)^{q-\gamma}\left|\frac{df(ry)}{dr}\right|^{q}
d\sigma(y)dr\right)^\frac{1}{q}.
\end{multline*}
This implies that
\begin{multline}\label{SLZ3}
\left(\int_{B(0,R)}\frac{|f-f_{R}|^{q}}{\left|\log\left|\log\frac{e R}{|x|}\right|\right|^{\gamma}
\left|\log\frac{e R}{|x|}\right|}\frac{dx}{|x|^{Q}}\right)^\frac{1}{q}\\
\leq\frac{q}{\gamma-1}\left(\int_{B(0,R)}|x|^{q-Q}
\left|\log\frac{e R}{|x|}\right|^{q-1}\left|\log\left|\log\frac{e R}{|x|}\right|\right|^{q-\gamma}\left|\frac{1}{|x|}\mathbb{E}f\right|
^{q}dx\right)^{\frac{1}{q}}.
\end{multline}
Now we consider the integrals in \eqref{SLZ2} restricted on $B(0,e R) \backslash B(0,R)$.
\begin{multline*}
\int_{B(0,e R)\backslash B(0,R)}\frac{|f-f_{R}|^{q}}{\left|\log\left|\log\frac{e R}{|x|}\right|\right|^{\gamma}
\left|\log\frac{e R}{|x|}\right|}\frac{dx}{|x|^{Q}}\\
=\int_{R}^{eR}\frac{1}{r(\log\frac{e R}{r})(\log((\log \frac{e R}{r})^{-1}))^{\gamma}}\int_{\wp}|f(r y)-f(R y)|^{q}d\sigma(y)dr\\
=-\frac{1}{\gamma-1}\left[\left(\log\left(\left(\log\frac{e R}{r}\right)^{-1}\right)\right)^{-\gamma+1}\int_{\wp}|f(r y)-f(R y)|^{q}d\sigma(y)\right]^{r=eR}_{r=R}\\
+\frac{1}{\gamma-1}\int_{R}^{eR}\left(\log\left(\left(\log\frac{e R}{r}\right)^{-1}\right)\right)^{-\gamma+1}
\frac{d}{dr}\int_{\wp}|f(r y)-f(R y)|^{q}d\sigma(y)dr\\
=\frac{q}{\gamma-1}\int_{R}^{eR}\left(\log\left(\left(\log\frac{e R}{r}\right)^{-1}\right)\right)^{-\gamma+1}
{\rm Re}\int_{\wp}|f(r y)-f(R y)|^{q-2}\\
\times(f(r y)-f(R y))\overline{\frac{df(ry)}{dr}}d\sigma(y)dr.
\end{multline*}
Here $\sigma$ is the Borel measure on $\wp, q-\gamma+1>0$, so that the boundary term at $r=R$ vanishes due to inequalities
\begin{multline*}
\log\left(\left(\log \frac{e R}{r}\right)^{-1}\right)=\int_{1}^{\left(\log \frac{e R}{r}\right)^{-1}}\frac{dt}{t} \\ \geq \left(\log\frac{e R}{r}\right)\left(\left(\log\frac{e R}{r}\right)^{-1}-1\right)=1-\log\frac{e R}{r}\geq\frac{r-R}{R}
\end{multline*}
and
$$|f(ry)-f(Ry)|\leq C (R-r)$$
for $R\leq r \leq eR$.
It follows that
\begin{multline*}
\int_{R}^{eR}\frac{1}{r(\log\frac{e R}{r})(\log((\log \frac{e R}{r})^{-1}))^{\gamma}}\int_{\wp}|f(r y)-f(R y)|^{q}d\sigma(y)dr\\
\leq \frac{q}{\gamma-1}\int_{R}^{eR}\left(\log\left(\left(\log\frac{e R}{r}\right)^{-1}\right)\right)^{-\gamma+1}\\
\times\int_{\wp}|f(r y)-f(R y)|^{q-1}\left|\frac{df(ry)}{dr}\right|d\sigma(y)dr
\\=\frac{q}{\gamma-1}
\int_{R}^{eR}\frac{1}
{r^{\frac{q-1}{q}}(\log\frac{e R}{r})^{\frac{q-1}{q}}
r^{-\frac{q-1}{q}}(\log\frac{e R}{r})^{-\frac{q-1}{q}}}
\\
\times \frac{1}{(\log((\log\frac{e R}{r})^{-1}))^{\frac{(q-1)\gamma}{q}}
(\log((\log \frac{e R}{r})^{-1}))^{\frac{\gamma-q}{q}}}\\
\times\int_{\wp}|f(r y)-f(R y)|^{q-1}\left|\frac{df(ry)}{dr}\right|d\sigma(y)dr.
\end{multline*}
By the H\"{o}lder inequality, we obtain
\begin{multline*}
\int_{R}^{eR}\frac{1}{r(\log\frac{e R}{r})(\log((\log \frac{e R}{r})^{-1}))^{\gamma}}\int_{\wp}|f(r y)-f(R y)|^{q}d\sigma(y)dr\\
\leq \frac{q}{\gamma-1} \left(\int_{R}^{eR} \int_{\wp}
\frac{|f(r y)-f(R y)|^{q}}{r(\log \frac{e R}{r})(\log((\log \frac{e R}{r})^{-1}))^{\gamma}}d\sigma(y)dr\right)^\frac{q-1}{q}\\
\times \left(\int_{R}^{eR} \int_{\wp}r^{q-1}\left(\log \frac{e R}{r}\right)^{q-1}\left(\log\left(\left(\log\frac{e R}{r}\right)^{-1}\right)\right)^{q-\gamma} \right.\\\left.
\times\left|\frac{df(r y)}{dr}\right|^{q}
d\sigma(y)dr\right)^{\frac{1}{q}}.
\end{multline*}
This implies that
\begin{multline}\label{SLZ4}
\left(\int_{B(0,e R) \backslash B(0,R)}\frac{|f-f_{R}|^{q}}{\left|\log\left|\log\frac{e R}{|x|}\right|\right|^{\gamma}
\left|\log\frac{e R}{|x|}\right|}\frac{dx}{|x|^{Q}}\right)^\frac{1}{q}
\leq\frac{q}{\gamma-1}\\\times \left(\int_{B(0,e R)\backslash B(0,R)}|x|^{q-Q}
\left|\log\frac{e R}{|x|}\right|^{q-1}\left|\log\left|\log\frac{e R}{|x|}\right|\right|^{q-\gamma}\left|\frac{1}{|x|}\mathbb{E}f\right|
^{q}dx\right)^{\frac{1}{q}}.
\end{multline}
The inequalities \eqref{SLZ3} and \eqref{SLZ4} imply \eqref{SLZ2}.
\end{proof}
Similarly, one can prove a dual inequality of \eqref{SLZ2} stated as follows.
\begin{prop} \label{SLZ04} Let $\mathbb{G}$ be a homogeneous group of homogeneous dimension $Q$ and
let $|\cdot|$ be any homogeneous quasi-norm on $\mathbb{G}$. Let $1<\gamma<\infty$ and $\max\{1,\gamma-1\}<q<\infty$. Then for all $f\in C_{0}^{\infty}(\mathbb{G})$ and for any $R>0$, the following inequality holds
\begin{multline}\label{SLZ05}
\left(\int_{B^{c}(0,R)}\frac{|f-f_{eR}|^{q}}{\left|\log\left|\log\frac{ R}{|x|}\right|\right|^{\gamma}
\left|\log\frac{R}{|x|}\right|}\frac{dx}{|x|^{Q}}\right)^{\frac{1}{q}}\\
\leq\frac{q}{\gamma-1}\left(\int_{B^{c}(0,R)}|x|^{q-Q}
\left|\log\frac{R}{|x|}\right|^{q-1}\left|\log\left|\log\frac{ R}{|x|}\right|\right|^{q-\gamma}\left|\frac{1}{|x|}\mathbb{E}f\right|
^{q}dx\right)^{\frac{1}{q}}.
\end{multline}
\end{prop}
Now let us prove \eqref{SLZ1} in Theorem \ref{SLZ0}.
\begin{proof}[Proof of Theorem \ref{SLZ0}] Using \eqref{SLZ05} with $R$ replaced by $e R$, we have
\begin{multline}\label{SLZ6}
\left(\int_{B^{c}(0,e R)}\frac{|f-f_{e^{2}R}|^{q}}{\left|\log\left|\log\frac{ e R}{|x|}\right|\right|^{\gamma}
\left|\log\frac{e R}{|x|}\right|}\frac{dx}{|x|^{Q}}\right)^{\frac{1}{q}}\\
\leq\frac{q}{\gamma-1}\left(\int_{B^{c}(0,e R)}|x|^{q-Q}
\left|\log\frac{e R}{|x|}\right|^{q-1}\left|\log\left|\log\frac{ e R}{|x|}\right|\right|^{q-\gamma}\left|\frac{1}{|x|}\mathbb{E}f\right|
^{q}dx\right)^{\frac{1}{q}}.
\end{multline}
Then from \eqref{SLZ2} and \eqref{SLZ6}, we obtain \eqref{SLZ1} for $f\in C_{0}^{\infty}(\mathbb{G})$.

Now we prove \eqref{SLZ1} for $f\in W^{1}_{0}L_{Q,q,\frac{q-1}{q},
\frac{q-\gamma}{q}}(\mathbb{G})$. We show first that \eqref{SLZ2} holds for $f\in W^{1}_{0}L_{Q,q,\frac{q-1}{q},
\frac{q-\gamma}{q}}(\mathbb{G})$. Let $\{f_{m}\}\subset C_{0}^{\infty}(\mathbb{G})$ be a sequence such that $f_{m}\rightarrow f$ in $W^{1}_{0}L_{Q,q,\frac{q-1}{q},
\frac{q-\gamma}{q}}(\mathbb{G})$ as $m\rightarrow \infty$ and almost everywhere by the definition \eqref{SLZdef1}. If we define
$$f_{R,m}(x):=\frac{f_{m}(x)-f_{m}(R\frac{x}{|x|})}{\left|\log\left|\log\frac{e R}{|x|}\right|\right|^{\frac{\gamma}{q}}\left|\log\frac{e R}{|x|}\right|^{\frac{1}{q}}},$$
then $\{f_{R,m}\}_{m \in \mathbb{N}}$ is a Cauchy sequence in $L^{q}(\mathbb{G};dx/|x|^{Q})$, which is a weighted Lebesgue space, since the inequality \eqref{SLZ2} holds for $f_{m}-f_{k}\in C_{0}^{\infty}(\mathbb{G})$. Consequently, there exists $g_{R}\in L^{q}(\mathbb{G};dx/|x|^{Q})$ such that $f_{R,m}\rightarrow g_{R}$ in $L^{q}(\mathbb{G};dx/|x|^{Q})$ as $m\rightarrow \infty$. From
\begin{multline*}
\left\{x \in \mathbb{G}\backslash\{0\}:f_{m}\left(R\frac{x}{|x|}\right)\not \rightarrow
 f\left(R\frac{x}{|x|}\right)\right\}\\
\subset \bigcup_{r>0}\left\{x \in \mathbb{G}\backslash\{0\}: f_{m}\left(r\frac{x}{|x|}\right)\not \rightarrow
 f\left(r\frac{x}{|x|}\right)\right\}\\ =\{x \in \mathbb{G}\backslash\{0\}:f_{m}(x)\not \rightarrow f(x)\},
 \end{multline*}
it follows that $f_{m}\left(R\frac{x}{|x|}\right)\rightarrow
 f\left(R\frac{x}{|x|}\right)$, that is, almost everywhere we obtain
 $$\frac{f(x)-f\left(R\frac{x}{|x|}\right)}{\left|\log\left|\log\frac{e R}{|x|}\right|\right|^{\frac{\gamma}{q}}\left|\log\frac{e R}{|x|}\right|^{\frac{1}{q}}}=g_{R}(x).$$
That is the inequality \eqref{SLZ2} holds for all $f\in W^{1}_{0}L_{Q,q,\frac{q-1}{q},
\frac{q-\gamma}{q}}(\mathbb{G})$.
In a similar way we obtain inequality \eqref{SLZ6} for all $f\in W^{1}_{0}L_{Q,q,\frac{q-1}{q},
\frac{q-\gamma}{q}}(\mathbb{G})$. Since the inequalities \eqref{SLZ2} and \eqref{SLZ6} hold for all $f\in W^{1}_{0}L_{Q,q,\frac{q-1}{q},
\frac{q-\gamma}{q}}(\mathbb{G})$, we obtain the inequality \eqref{SLZ1} for $f\in W^{1}_{0}L_{Q,q,\frac{q-1}{q},
\frac{q-\gamma}{q}}(\mathbb{G})$.

Now let us prove the sharpness of the constant $\frac{q}{\gamma-1}$ in \eqref{SLZ1}.
From \eqref{SLZ1} for all $f\in W^{1}_{0}L_{Q,q,\frac{q-1}{q},\frac{q-\gamma}{q}}(B(0,R))$, one obtains
\begin{multline}\label{SLZ7}
\left(\int_{B(0,R)}\frac{|f(x)|^{q}}{\left|\log\left|\log\frac{ e R}{|x|}\right|\right|^{\gamma}
\left|\log
\frac{e R}{|x|}\right|}\frac{dx}{|x|^{Q}}\right)^{\frac{1}{q}}\\
\leq\frac{q}{\gamma-1}\left(\int_{B(0,R)}|x|^{q-Q}
\left|\log\frac{e R}{|x|}\right|^{q-1}\left|\log\left|\log\frac{ e R}{|x|}\right|\right|^{q-\gamma}\left|\frac{1}{|x|}\mathbb{E}f\right|
^{q}dx\right)^{\frac{1}{q}}.
\end{multline}
Therefore, it is enough to show the sharpness of the constant $\frac{q}{\gamma-1}$ in \eqref{SLZ7}. As in the abelian case (see \cite[Section 3]{MOW15c}), we consider a sequence of functions $\{f_{\ell}\}$ for large $\ell\in \mathbb{N}$ defined by
$$f_{\ell}(x):=\begin{cases}
(\log(\log(\ell eR)))^{\frac{\gamma-1}{q}}, \;\;\;{\rm when}\;\;\;|x|\leq\frac{1}{\ell},\\
\left(\log\left(\log \frac{eR}{|x|} \right)\right)^{\frac{\gamma-1}{q}}, \;\;\;{\rm when}\;\;\;\frac{1}{\ell}\leq|x|\leq\frac{R}{2},\\
(\log(\log(2e)))^{\frac{\gamma-1}{q}}\frac{2}{R}(R-|x|), \;\;\;{\rm when}\;\;\;\frac{R}{2}\leq|x|\leq R.
\end{cases}$$
It is easy to see that $f_{\ell}\in W^{1}_{0}L_{Q,q,\frac{q-1}{q},
\frac{q-\gamma}{q}}(B(0,R))$. Letting $\widetilde{f}_{\ell}(r):=f_{\ell}(x)$ with $r=|x|\geq0$, we obtain
$$\frac{d}{dr}\widetilde{f}_{\ell}(r)=\begin{cases}
0, \;\;\;{\rm when}\;\;\;r<\frac{1}{\ell},\\
-\frac{\gamma-1}{q}r^{-1}\left(\log\left(\log\frac{eR}{r}\right)\right)^{\frac{\gamma-1}{q}-1}
(\log\frac{eR}{r})^{-1}, \;\;\;{\rm when}\;\;\;\frac{1}{\ell}<r<\frac{R}{2},\\
-\frac{2}{R}(\log(\log(2e)))^{\frac{\gamma-1}{q}}, \;\;\;{\rm when}\;\;\;\frac{R}{2}< r<R.
\end{cases}$$
Denoting by $|\wp|$ the $Q-1$ dimensional surface measure of the unit sphere,
by a direct calculation we have
\begin{multline*}
\int_{B(0,R)}|x|^{q-Q}
\left|\log\frac{e R}{|x|}\right|^{q-1}\left|\log\left|\log\frac{ e R}{|x|}\right|\right|^{q-\gamma}\left|\frac{1}{|x|}\mathbb{E}f_{\ell}(x)\right|
^{q}dx\\
=|\wp|\int_{0}^{R}r^{q-1}\left|\log\frac{e R}{r}\right|^{q-1}\left|\log\left|\log\frac{ e R}{r}\right|\right|^{q-\gamma}\left|\frac{d}{dr}\widetilde{f}_{\ell}(r)\right|^{q}dr\\
=|\wp|\left(\frac{\gamma-1}{q}\right)^{q}\int_{\frac{1}{\ell}}^{\frac{R}{2}}r^{-1}\left(\log\frac{e R}{r}\right)^{-1}\left(\log\left(\log\frac{ e R}{r}\right)\right)^{-1}dr\\
+(\log(\log(2e)))^{\gamma-1}\left(\frac{2}{R}\right)^{q}|\wp|
\int_{\frac{R}{2}}^{R}r^{q-1}\left(\log\frac{e R}{r}\right)^{q-1}\left(\log\left(\log\frac{ e R}{r}\right)\right)^{q-\gamma}dr\\=
-|\wp|\left(\frac{\gamma-1}{q}\right)^{q}\int_{\frac{1}{\ell}}^{\frac{R}{2}}\frac{d}{dr}\left(\log\left(\log\left(\log \frac{ e R}{r}\right)\right)\right)
\\+(\log(\log(2e)))^{\gamma-1}\left(\frac{2}{R}\right)^{q}|\wp|
\int_{\frac{R}{2}}^{R}r^{q-1}\left(\log\frac{e R}{r}\right)^{q-1}\left(\log\left(\log\frac{ e R}{r}\right)\right)^{q-\gamma}dr\\
\end{multline*}
\begin{equation} \label{SLZ8}
=:|\wp|\left(\frac{\gamma-1}{q}\right)^{q}(\log(\log(\log \ell e R))-\log(\log(\log 2 e)))+|\wp|C_{\gamma,q},
\end{equation}
where $$C_{\gamma,q}:=(2 e)^{q}(\log(\log(2e)))^{\gamma-1}\int_{0}^{(\log(\log(2e)))}s^{q-\gamma}
e^{q(s-e^{s})}ds.
$$
By the assumption $q-\gamma+1>0$, we get $C_{\gamma,q}<+\infty$. On the other hand, we see
\begin{multline*}
\int_{B(0,R)}\frac{|f_{\ell}(x)|^{q}}{\left|\log\left|\log\frac{ e R}{|x|}\right|\right|^{\gamma}
\left|\log\frac{e R}{|x|}\right|}\frac{dx}{|x|^{Q}}
=|\wp|\int_{0}^{R}\frac{|\widetilde{f}_{\ell}(r)|^{q}}{\left|\log\left|\log\frac{ e R}{r}\right|\right|^{\gamma}
\left|\log\frac{e R}{r}\right|}\frac{dr}{r}\\
=|\wp|(\log(\log(\ell eR)))^{\gamma-1}
\int_{0}^{\frac{1}{\ell}}r^{-1}
\left(\log\frac{e R}{r}\right)^{-1}\left(\log\left(\log\frac{ e R}{r}\right)\right)^{-\gamma}dr\\
+|\wp|\int_{\frac{1}{\ell}}^{\frac{R}{2}}r^{-1}
\left(\log\frac{e R}{r}\right)^{-1}\left(\log\left(\log\frac{ e R}{r}\right)\right)^{-1}dr\\
+|\wp|(\log(\log(2e)))^{\gamma-1}\left(\frac{2}{R}\right)^{q}
\int_{\frac{R}{2}}^{R}r^{-1}
(R-r)^{q}\left(\log\frac{e R}{r}\right)^{-1}\\ \times \left(\log\left(\log\frac{ e R}{r}\right)\right)^{-\gamma}dr
\end{multline*}
\begin{equation}\label{SLZ9}
=:\frac{|\wp|}{\gamma-1}+|\wp|(\log(\log(\log(\ell eR))-\log(\log(\log(2e)))
+|\wp|C_{R,\gamma,q},
\end{equation}
where
\begin{multline*}
C_{R,\gamma,q}:=(\log(\log(2e)))^{\gamma-1}\left(\frac{2}{R}\right)^{q}\\
\times\int_{\frac{R}{2}}^{R}r^{-1}
(R-r)^{q}\left(\log\frac{e R}{r}\right)^{-1}\left(\log\left(\log\frac{ e R}{r}\right)\right)^{-\gamma}dr.
\end{multline*}
The inequality $\log(\log\frac{eR}{r})\geq \frac{R-r}{R}$ for all $r\leq R$ and the assumption $q-\gamma>-1$, imply $C_{R,\gamma,q}<+\infty$. Then, by \eqref{SLZ8} and \eqref{SLZ9}, we have
\begin{multline*}
\int_{B(0,R)}|x|^{q-Q}
\left|\log\frac{e R}{|x|}\right|^{q-1}\left|\log\left|\log\frac{ e R}{|x|}\right|\right|^{q-\gamma}\left|\frac{1}{|x|}\mathbb{E}f_{\ell}\right|
^{q}dx\\
\times \left(\int_{B(0,R)}\frac{|f_{\ell}(x)|^{q}}{\left|\log\left|\log\frac{ e R}{|x|}\right|\right|^{\gamma}
\left|\log\frac{e R}{|x|}\right|}\frac{dx}{|x|^{Q}}\right)^{-1}\rightarrow \left(\frac{\gamma-1}{q}\right)^{q}
\end{multline*}
as $\ell \rightarrow \infty$, which implies that the constant $\frac{q}{\gamma-1}$ in \eqref{SLZ7} is sharp.
\end{proof}


\end{document}